\documentclass{article}

\usepackage[utf8]{inputenc}
\usepackage[all,arc]{xy}
\usepackage{enumerate}
\usepackage[top=1in, bottom=1.25in, left=1.25in, right=1.25in]{geometry} 
\usepackage{parskip}
\usepackage{braket}
\usepackage{tikz-cd}
\usepackage{verbatim}
\usepackage{tabu}
\usepackage{ggmaths}

\usepackage{biblatex}
\usepackage{hyperref}
\usepackage{xurl}
\hypersetup{breaklinks=true}
\title{Good election rules with more than three candidates are Borda}
\author{\makeauthorentry}

\begin{document}

\maketitle

\begin{abstract}

Arrow \autocite{arrow} proved that for three or more candidates, the IIA condition is enough to forbid all non-dictatorial election rules (or Social Welfare Functions). Maskin \autocite{maskin} introduced the weaker MIIA condition, which permits the ``Borda'' election rules where each voter assigns points linearly to each candidate according to their order of preference. However, in previous work \autocite{paper1} we demonstrated that there exist Social Welfare Functions between three candidates and satisfying the MIIA condition which are far from being Borda rules.

We demonstrate that this phenomenon is unique to the case of three candidates. As soon as a fourth candidate is introduced, and indeed for any larger number of candidates, the only good election rules are the unweighted Borda rules.

\end{abstract}

\section{Introduction}

As in \autocite{paper1}, we let a set of voters be $V$ and the set of candidates $C$, with $|C| = k < \infty$. We always label these candidates as $\set{c_i}_{i=1}^k$, and we treat the indices modulo $k$. Throughout this paper, we take $k > 3$.

We will consider two cases for the size of $V$. Either it will be finite, in which case we say $|V| = n$, or it will be infinite, in which case we will require that it be equipped with a $\sigma$-algebra $\Sigma$ and a measure $\nu$. SWFs on unmeasurable electorates are not interesting within the context of these investigations due to the need to ``count'' or ``measure'' different sets of voters in order to compare their influences in choosing the final ranking.

In the infinite case, we require that the measure is finite (i.e. $\nu(V) < \infty$), and indeed we normalise to $\nu(V)=1$. Again, this is the critical case because if $\nu(V) = \infty$ then we might have several competing groups all with measure $\infty$, making it impossible to compare them meaningfully. We also require that $\nu$ is atomless (also known as diffuse); that is to say, there is no set $A \in \Sigma$ with $\nu(A)>0$ such that for any $B \subset A$ in $\Sigma$ we have either $\nu(B) = 0$ and $\nu(B) = \nu(A)$. Sierpinski \autocite{sierpinski} showed that this allows us to find subsets of arbitrary measure within any set of larger measure.

We define $\mathfrak{R}$ to be the set of possible votes $r$ that voter $v$ can register. Each $r$ is a total ordering of $C$, and we will treat $r$ as a bijection $r: C \hookrightarrow \set{0, \dots, |C|-1}$, where $r^{-1}(0)$ is the lowest ranked candidate, and $r^{-1}(|C|-1)$ is the highest ranked. We will sometimes write $r$ simply as a list of candidates; for example, $c_1 c_2 c_3$ corresponds to a voter's ballot placing candidate $c_1$ first, followed by $c_2$ and then $c_3$. This corresponds to $r: c_i \mapsto 3-i$.

We define $\bar{E} \subset \mathfrak{R}^V$ to be the set of possible elections. For $|V|$ finite, this is simply $\bar{E} = \mathfrak{R}^V$; if $|V|$ is infinite then $\bar{E} = \set{e: V \rightarrow \mathfrak{R} | e \text{ measurable}}$. Note that $C$ is always finite, so $\mathfrak{R}$ is finite, and we endow it with the discrete $\sigma$-algebra. In other papers we have considered SWFs which restrict the set of legal elections to some $E \subset \bar{E}$; in this paper, we always have $E = \bar{E}$. This is called the ``unrestricted domain'', and a SWF with this property is said to fulfil condition U. In \autocite{paper1} we showed that this is an increasing, intermediate, separable ballot domain, so we will be able to make use of various lemmas here which apply to such domains.

An election result $\succ$ is a weak ordering of $C$ - in other words, an equivalence relation of ``tying'', together with a total ordering of ``winning'' between the equivalence classes. We call the set of all possible weak orderings $\mathfrak{R}_=$. In the most general sense, a SWF is any function $F: E \rightarrow \mathfrak{R}_=$.

For $r \in \mathfrak{R}$ we let $\pi_{i,j}(r) = r(c_i)-r(c_j)$, and we define

\begin{equation}
D_k = \pi_{i,j}(\mathfrak{R}) = \set{1-k, 2-k, \dots, -1, 1, 2, \dots, k-1}
\end{equation}

Then for an election $e \in E$ we let $\pi_{i,j}(e)(v) = \pi_{i,j}(e(v))$, so $\pi_{i,j}(e)$ is the ``relative election'' comparing the positions of candidates $c_i$ and $c_j$ for each voter. Then $\pi_{i,j}(E) = D_k^V$ for $V$ finite, and for $V$ infinite $\pi_{i,j}(E)$ is the set of measurable functions from $V$ to $D_k$. In either case we let $\pi_{i,j}(E) = A$. We define the partial order on $A$ given by $a_1 \geq a_2$ if and only if $a_1(v) \geq a_2(v)$ for all $v$.

Finally, for an election result $\succ$ we write $\pi_{i,j}(\succ) = W$ if $c_i \succ c_j$, $\pi_{i,j}(\succ)=L$ if $c_j \succ c_i$, and $\pi_{i,j}(\succ) = T$ otherwise. Here $W$ stands for ``win'', $T$ for ``tie'' and $L$ for ``loss''. For convenience we impose the natural total order on these three symbols: $W > T > L$.  Similarly, we will sometimes abuse notation by writing $W = -L$, $L = -W$ and $T=-T$.

\section{Conditions on SWFs}\label{section:conditions}

We now define some conditions on SWFs. The first is the critical Modified IIA condition introduced by Maskin in \autocite{maskin}, as an alternative to Arrow's \autocite{arrow} stricter IIA condition. This definition is simplified in the case of the unrestricted domain; for the more general definition see \autocite{paper1}.

\begin{defn}[Modified Irrelevance of Independent Alternatives - MIIA]
A SWF $F$ satisfies the MIIA condition if, for two candidates $c_i$ and $c_j$, there exists a ``relative SWF'': a function $f_{i,j}: A \rightarrow \set{W, T, L}$ such that for all $e \in E$ we have $f_{i,j}(\pi_{i,j}(e)) = \pi_{i,j}(F(e))$.
\end{defn}

Intuitively, if for two elections, each voter places the same order between candidates $c_1$ and $c_2$ with the same number of other candidates between them, then the results of those two elections must give the same ranking between $c_1$ and $c_2$.

The anonymity condition is that the SWF is unaffected by which voters submit which ballots, and instead is determined only by \emph{how many} voters (either by count or by measure) vote in a given way. In light of this, we define $\tau: E \rightarrow \mathbbm{R}_{\geq 0}^{\mathfrak{R}}$ by $\tau(e)(r) = \nu(e^{-1}(r))$ (or for $V$ finite $\tau(e)(r) = |e^{-1}(r)|$), and we let $E' = \Image(\tau)$.

\begin{defn}[Anonymity - A]
A SWF $F$ satisfies anonymity if there exists a function $F': E' \rightarrow \mathfrak{R}_=$ such that $F = F' \circ \tau$.
\end{defn}

We can think of $F'$ as the ``weight-based'' SWF. We proved in \autocite{paper1} that $F$ and $F'$ determine one another.

Similarly, we let $\tau_A: A \rightarrow \mathbbm{R}_{\geq 0}^{D_k}$ by $\tau_A(a)(i) = \nu(a^{-1}(i))$, or $|a^{-1}(i)|$ for $V$ finite, and let $A' = \Image(\tau_A)$. We also define $\varrho: A' \rightarrow A'$ by $\varrho(\vec{\alpha})_k = \alpha_{-k}$; it is immediate that $\varrho \circ \tau_A \circ \pi_{i,j} = \tau_A \circ \pi_{j,i}$. The following is a result from \autocite{paper1}:

\begin{lem}\label{amiia}
Let $F$ be a SWF on a ballot domain satisfying MIIA and A. Then for any two candidates $c_i$ and $c_j$ we can define a weight-based relative SWF $f'_{i,j}$ such that $f_{i,j} = f'_{i,j} \circ \tau_A$.
\end{lem}

Again, $f_{i,j}$ and $f'_{i,j}$ uniquely determine one another for each $i \neq j$.

We define a partial order on $A'$ corresponding to the order we already have on $A$. We say that $\vec{\alpha} \geq \vec{\beta}$ if, for all $m \in D_k$, we have

\begin{equation}
\sum_{n\geq m} \alpha_n \geq \sum_{n\geq m} \beta_n
\end{equation}

In other words, $\vec{\alpha}$ majorises $\vec{\beta}$. Clearly if $a \geq b \in A$ then $\tau_A(a) \geq \tau_A(b)$.

For $|V|$ finite, we showed that anonymity is equivalent to the condition that any permutation of the voters fixes $F$. In \autocite{paper2} we introduced a weaker anonymity condition, called ``transitive anonymity'', requiring $F$ to be invariant under a group $G$ acting transitively on $V$, and studied SWFs with transitive anonymity and $k=3$. In a forthcoming paper we will return to this weaker version of anonymity in the case of $k>3$.

As well as symmetry between voters we expect symmetry between candidates. For a permutation $\rho: C \rightarrow C$ we can apply $\rho$ to $r \in \mathfrak{R}$ or to $\succ \in \mathfrak{R}_=$ in the natural way: we define $\rho(r) = r \circ \rho$ and $\rho(\succ)$ to be the unique weak ordering such that for all $c_i \neq c_j \in C$ with $\rho(c_i)=c_{i'}$ and $\rho(c_j)=c_{j'}$ we have $\pi_{i',j'}(\rho(\succ)) = \pi_{i,j}(\succ)$.

\begin{defn}[Neutrality - N]
A SWF satisfies neutrality if the candidates are treated symmetrically. That is to say, for any permutation $\rho: C \rightarrow C$ and for all $e$ such that $\rho \circ e \in E$ we have $F(\rho \circ e) = \rho(F(e))$.
\end{defn}

We proved in \autocite{paper1} that for a SWF satisfying MIIA, A and N, all the functions $f'_{i,j}$ are identical. There, the statement applied to SWFs on any separable domain; here we restate the Lemma for the specific case of the unrestricted domain.

\begin{lem}\label{namiia}
Let $F$ be a SWF satisfying U, MIIA, A and N. Then for any candidates $c_i \neq c_j$ and $c_{i'} \neq c_{j'}$, $f'_{i,j} = f'_{i',j'}$.
\end{lem}

We call this shared relative SWF $f'$.

Two more conditions on SWFs, Positive Responsiveness and the Pareto condition, describe senses in which the SWF should be ``increasing'' -- for each candidate, doing well in a voter's ballot should only benefit them in the final ranking.

\begin{defn}[Pareto - P]
A SWF satisfying MIIA also satisfies the Pareto condition if, whenever all voters prefer candidate $c_i$ to $c_j$, the same must be the case in the final result. Formally, for $a \in A_{i,j}$ with $a(v) > 0$ for all $v$, we have $f_{i,j}(a)=W$.
\end{defn}

\begin{defn}[Positive Responsiveness]
A SWF satisfying MIIA also satisfies the PR condition if, for any two candidates $c_i$ and $c_j$, and for any two relative elections $a_1, a_2 \in A_{i,j}$ where $a_1 \geq a_2$, we have $f_{i,j}(a_1) \geq f_{i,j}(a_2)$.
\end{defn}

We ignore the distinction between conditions PR and PRm from \autocite{paper1} because we always impose the U condition, so our domain is always increasing.

Finally, we recall the definition of consistency in the context of $3$-tuples of results from relative SWFs:

\begin{defn}[Consistent multisets]
A multiset of cardinality three, with elements drawn from $\set{W,T,L}$, is consistent if it is one of the following:
\begin{enumerate}
\item $\set{W,W,L}$
\item $\set{W,T,L}$
\item $\set{W,L,L}$
\item $\set{T,T,T}$
\end{enumerate}

Otherwise the multiset is ``inconsistent''.
\end{defn}

We extend the definition of consistency to functions $f': A' \rightarrow \set{W,T,L}$ as follows:

\begin{defn}[Consistent function]
A function $f': A' \rightarrow  \set{W,T,L}$ is consistent if for any weight-based election $\varepsilon \in E'$:
\begin{enumerate}
\item\label{consistent:twoway} for all $c_i \neq c_j$, $f' \circ \pi_{i,j}(\varepsilon) = -f'\circ \pi_{j,i}(\varepsilon)$
\item\label{consistent:threeway} for all $c_i \neq c_j \neq c_k$, $\set{f' \circ \pi_{i,j}(\varepsilon), f' \circ \pi_{j,k}(\varepsilon), f' \circ \pi_{k,i}(\varepsilon)}$ is a consistent multiset
\end{enumerate}
\end{defn}

This allows us to simplify the task of determining SWFs $F$ by searching instead for the corresponding weight-based relative SWFs, according to the following Lemma:

\begin{lem}\label{Ftofprime}
Given a SWF $F$ satisfying U, N, A and MIIA, the function $f'$ given by Lemma \ref{namiia} is consistent.
\end{lem}

Note that the converse of this statement is immediate: given a consistent function $f'$ we can construct $F$ simply by defining $F(e)$ to be the unique ranking in $\mathfrak{R}_=$ to accord with all pairwise rankings given by $f'(\tau_A(\pi_{i,j}(e)))$.

\begin{proof}
Take any $\varepsilon \in E'$, and take any $e \in E$ with $\tau(e)=\varepsilon$. Then for all $c_i$ and $c_j$ we have

\begin{equation}
\begin{aligned}
f' \circ \pi_{i,j}(\varepsilon) &= f' \circ \pi_{i,j} \circ \tau(e) \\
&= f' \circ \tau_A \circ \pi_{i,j}(e) \\
&= f \circ \pi_{i,j}(e) \\
&= \pi_{i,j} \circ F(e) 
\end{aligned}
\end{equation}

Therefore, the conditions for $f'$ being consistent are as follows. Condition \ref{consistent:twoway} states that for $c_i \neq c_j$ we have $\pi_{i,j}(\succ) = - \pi_{j,i}(\succ)$, which is trivially true; condition \ref{consistent:threeway} states that for any $c_i \neq c_j \neq c_k$ we have $\set{\pi_{i,j}(\succ), \pi_{j,k}(\succ), \pi_{k,i}(\succ)}$ a consistent multiset; and again this follows immediately from the definition of a weak total order (for a careful examination of the cases, see the proof of Lemma 2.22 in \autocite{paper1}).
\end{proof}

Finally we recall the Borda rules. These are SWFs in which each voter awards points to candidates equal to a scalar multiple of their placement in that voter's ranking. In \autocite{paper1} and \autocite{paper2} we addressed weighted Borda rules, in which different voters had different corresponding scalars. In this paper, all of our results will establish that under certain conditions, the only legal SWFs are the three unweighted Borda rules:

\begin{defn}[(Unweighted) Borda rule]
For $w=\pm1$ or $w=0$, the unweighted Borda rule $B_w$ is defined as follows. We define

\begin{equation}
\begin{aligned}
b_w(c_i,e) &= w\int_V e(v)(c_i)\diff \nu(v) \quad & |V| = \infty \\
b_w(c_i,e) &= w\sum_{v \in V} e(v)(c_i) \quad & |V| < \infty
\end{aligned}
\end{equation}

Now we let $B_w(e)$ be the unique weak total ordering corresponding to the ordering of the values $b_w(c_i,e)$; in other words, for all $i$ and $j$ we have $c_i \succ c_j$ if and only if $b_w(c_i,e) > b_w(c_j,e)$.
\end{defn}

We call $B_1$ the positive unweighted Borda rule; $B_{-1}$ is the negative unweighted Borda rule, and $B_0$ is the tie rule (because $b_0(c_i,e)=0$ for all $c_i$ and $e$, so $B_0$ always gives a $k$-way tie between all candidates).

We also define $d_w(c_i,c_j,e) = b_w(c_i,e)-b_w(c_j,e)$, and note that

\begin{equation}\label{eq:relativeborda}
\begin{aligned}
d_w(c_i,c_j,e) &= b_w(c_i,e)-b_w(c_j,e) \\
&= w\int_V e(v)(c_i)-e(v)(c_j) \diff \nu(v) \\
&= w\int_V \pi_{i,j}(e)(v) \diff \nu(v) \quad |V| = \infty
\end{aligned}
\end{equation}

The equivalent for $|V|$ finite is

\begin{equation}
d_w(c_i,c_j,e) = w\sum_{v \in V} \pi_{i,j}(e)(v)
\end{equation}

Thus $d_w$ is a function of $\pi_{i,j}(e)$, and we can write $d_w(c_i,c_j,e) = d_w(\pi_{i,j}(e))$. In fact, writing $\alpha = \tau_A(a)$, then

\begin{equation}\label{eq:relativeweightbasedborda}
\begin{aligned}
d_w(a) &= w\int_V a(v) \diff \nu(v) \\
&= w\sum_{d \in D_k} \int_{v \in a^{-1}(d)} d \diff \nu(v) \\
&= w\sum_{d \in D_k} d \alpha_d
\end{aligned}
\end{equation}

Consequently we can treat $d_w$ as a function on values $\alpha$ in $A'$. The results in this paper establish when a SWF $F$ must be unweighted Borda rules; we do this by demonstrating that the corresponding function $f'$ can be factored as $f' = g \circ d_w$, and then that the sign of $d_w$ is enough to determine $f'$.

\begin{defn}[Borda-determined]
We say that a SWF $F$ satisfying U, A, N and MIIA is Borda-determined if, for $f'$ corresponding to $F$ as defined in Lemma \ref{namiia}, we can write $f' = g \circ d_1$.
\end{defn}

We let $\varphi: \mathbbm{R} \rightarrow \set{W,T,L}$ be defined by

\begin{equation}
\varphi(d) = \begin{cases}
W & d>0 \\
T & d=0 \\
L & d<0
\end{cases}
\end{equation}

Then $B_1$ is Borda-determined with $g=\varphi$. Similarly $B_{-1}$ is Borda-determined with $g = -\varphi$, and $B_0$ is Borda-determined with $g\equiv T$. Clearly each $g$ uniquely determines a Borda-determined SWF $F$.

\section{Tools for evaluating relative weight-based SWFs}\label{section:tools}

Repeatedly in this paper, we will be presented with a relative weight-based SWF $f'$, and we will want to evaluate it for as many vectors $\vec{\alpha} \in A'$ as possible. In order to do so efficiently, we will define a structure called a transitive election table, which will allow us to exploit the consistency condition for $f'$ given in Lemma \ref{Ftofprime} to find vectors $\vec{\alpha}$ for which $f'(\vec{\alpha})$ must evaluate to $T$, or given a vector with $f'(\vec{\alpha})=X$, to find a related vector $\vec{\alpha}_1$ for which $f'(\vec{\alpha})=X$ as well. All of these results rely on the fact that we are in the unrestricted domain, but they apply in both the case of $V$ finite and of $V$ infinite. For brevity we define the measure $\nu$ on $V$ finite to be the counting measure (note that this means $\nu(V)=1$ in the infinite case but $\nu(V)=n$ in the case of $n$ voters).

We begin with the following intuitive lemma:

\begin{lem}\label{constructingthreewayvotes}
For $d_1, d_2 \in D_k$ with $d_1+d_2 \in D_k$, there exists $r=r(d_1,d_2) \in \mathfrak{R}$ such that $\pi_{1,2}(r) = d_1$ and $\pi_{2,3}(r) = d_2$ (and therefore $\pi_{1,3}(r) = d_1+d_2$).
\end{lem}

\begin{proof}
We define a function $g$ taking $1$ to $0$, $2$ to $d_1$ and $3$ to $d_1+d_2$. $g$ is injective as $d_1, d_2, d_1+d_2 \neq 0$. Now let the maximum value taken by this function be $M \in D_k \cup \set{0}$. We now define $g' = g + k-1 - M$. Clearly for all $x$ we have $g'(x) \leq k-1$. Moreover, the differences between pairs of values of $g$ are $|d_1|$, $|d_2|$ and $|d_1+d_2|$ which are all in $D_k$ so less than $k-1$; thus $g'(x) >0$ for all $x$. Therefore we can construct $r \in \mathfrak{R}$ with $r(c_i) = g'(i)$ for $i=1,2,3$, and this satisfies the equations required.
\end{proof}

This lemma describes how two pairwise results can be chained together to give a third, in the context of one voter. Next, we extend this by considering under which circumstances two full relative elections can be chained together to form a third.

\begin{lem}\label{constructingthreewayelections}
For vectors $\vec{\alpha}_1, \vec{\alpha}_2 \in A'$, if there exist $\vec{a}_1, \vec{a}_2 \in A$ such that $\tau_A(\vec{a}_i)=\vec{\alpha}_i$ respectively, and $\vec{a}_1 + \vec{a}_2 \in A$, then we can construct a weight-based election $\vec{\varepsilon}\in E'$ for which $\pi_{i,i+1}(\varepsilon) = \vec{\alpha}_i$ for $i=1,2$.
\end{lem}

\begin{proof}
For each $(d_1,d_2) \in D_k^2$ with $d_1+d_2 \in D_k$, we let

\begin{equation}
\begin{aligned}
\varepsilon_{r(d_1,d_2)} = \nu(\set{v | (d_1,d_2) = (\vec{a}_1(v),\vec{a}_2(v))})
\end{aligned}
\end{equation}

Then $\varepsilon \in \mathbb{R}^{\mathfrak{R}}$, and $\sum_r \varepsilon_r = 1$. By considering the set of $r \in \mathfrak{R}$ for which $\pi_{1,2}(r)=d_1$ we observe that $\pi_{1,2}(\varepsilon)(d_1) = \alpha_1(d_1)$, and similarly by considering the set of $r \in \mathfrak{R}$ for which $\pi_{2,3}(r)=d_2$ we observe that $\pi_{2,3}(\varepsilon)(d_2) = \alpha_2(d_2)$, as required.
\end{proof}

Applying $f'$ to these vectors $\vec{\alpha}_1$ and $\vec{\alpha}_2$ gives information about $f(\tau_A(\vec{a}_1+\vec{a}_2))$:

\begin{lem}\label{transitivethreeway}
Let $\vec{\alpha}_1$ and $\vec{\alpha}_2$ satisfy the condition of Lemma \ref{constructingthreewayelections}, and let $\vec{\alpha}_3 = \tau_A(\vec{a}_1+\vec{a}_2)$. Suppose that for some $X \in \set{W,T,L}$ we have $\set{X} \subset \set{f(\vec{\alpha}_1), f(\vec{\alpha}_2)} \subset \set{X,T}$. Then $f(\vec{\alpha}_3) = X$.
\end{lem}

\begin{proof}
Lemma \ref{constructingthreewayelections} tells us that there exists an election $\vec{\varepsilon} \in E'$ for which $\pi_{1,2}(\varepsilon)$, $\pi_{2,3}(\varepsilon)$ and $\pi_{3,1}(\varepsilon)$ are $\vec{a}_1$, $\vec{a}_2$ and $-\vec{a}_3$ respectively. Then $f'(\vec{a}_1),f'(\vec{a}_2),f'(-\vec{a}_3)$ form a consistent multiset. If the first two values are $T$ then the third must be $T$; if the first two are a $W$ and a $T$ or two $W$s then then third must be $L$, and vice versa.
\end{proof}

Now we will define a structure allowing us to repeatedly apply Lemma \ref{transitivethreeway} to compute $f'$ on more vectors in $A'$. In preparation for this, we define $\vec{\alpha}(w,m)$ for $m \in D_k^t$ and $w \in \mathbbm{R}_{\geq 0}^t$ such that $\sum_i w_i = 1$ and, in the case of $V$ finite, $w_i \in \mathbbm{Z}$. Then $\vec{\alpha}(w,m) \in A'$ is given by

\begin{equation}
\begin{aligned}
\alpha_d = \sum_{j: m_j = d} w_j
\end{aligned}
\end{equation}

In other words, each pair $w_i$ and $m_i$ tell us that $w_i$ voters have cast the relative vote $m_i$. We are now ready to define a Transitive Election Table:

\begin{defn}[Transitive Election Tables]
An $X$-transitive election table for $X \in \set{W,T,L}$ is a triple $(w,P,M)$ where:
\begin{itemize}
\item $w$ is a vector of $t$ non-negative values $w_1, w_2, \dots w_t$ with $\sum_i w_i = \nu(V)$; and for $V$ finite, $w_i \in \mathbbm{Z}$
\item $P$ describes the order in which a sequence of $m$ atoms $\tuple{x_1, \dots, x_m}$ are combined using a binary operation; it is written using parentheses, where any lowest-level pair of parentheses contains two factors, and can inductively be replaced by a new atom. So, for example, $x_1(x_2x_3)$ and $(x_1x_2)x_3$ are the two possible $P$ for $m=3$.
\item $M$ is a matrix in $D_k^{m \times t}$
\end{itemize}
satisfying the following properties:
\begin{itemize}
\item for each row of $M$, labelled $M_i \in D_k^t$, we have $f'(\vec{\alpha}(w,M_i)) \in \set{X, T}$
\item for at least one row of $M$ we have $f(\vec{\alpha}(w,M_i)) = X$;
\item for any pair of matching parentheses in $P$, containing the factors $x_i$ through to $x_j$, we have
\begin{equation}\begin{aligned}
\sum_{s=i}^j M_s \in D_k^t
\end{aligned}\end{equation}
\item we also have
\begin{equation}\begin{aligned}
M_0 = \sum_{s=1}^m M_s \in D_k^n \cup \set{\vec{0}}
\end{aligned}\end{equation}
\end{itemize}
\end{defn}

The following corollaries of Lemma \ref{transitivethreeway} follow immediately from applying the Lemma iteratively to parenthesised sets of rows in a Transitive Election Table:

\begin{cor}\label{transitiveelectiontable}
For an $X$-transitive election table $(w,M,P)$, we have $f'(\vec{\alpha}(w, M_0))= X$ if $M_0$ is in $D_k^n$.
\end{cor}

\begin{cor}\label{tyingelectiontable}
For an $X$-transitive election table $(w,M,P)$ with $M_0 = \vec{0}$ we have $X=T$.
\end{cor}

We call this a ``tying election table''

We will write $w$ and $M$ as one table, where the top row gives the values of $w$ in parentheses, and subsequent rows are the rows of $M$, with a horizontal line after the $w$ row. Sometimes several entries in $w$ will be equal, in which case we will tend to show them adjacent to one another and only write the value of $w$ once, with vertical lines demarcating the columns for which that value of $w$ applies. So for example, a vector corresponding to $(k,0, k, 0, 2k, 0) \in A_4$ might be written as

\begin{equation}\begin{aligned}
\left(\begin{array}{c c | c}
(k) & & (2k) \\ \hline
3 & 1 & -2
\end{array}\right)
\end{aligned}\end{equation}

Sometimes a transitive election table will have $w = \mathbbm{1}$, in which case we will omit it altogether. We will occasionally transpose the table for readability if $t$ is large. Finally, note that we need not define $P$ for a transitive election table with $m=2$ or for a tying election table with $m=3$.

\section{Summary of results}

In section \ref{InfUAk} we deal with the case of an infinite set of voters. In this context we prove Theorem \ref{IUAk}, which states that under the U, MIIA, A and N conditions, our SWF must be Borda-determined. This theorem, combined with a measurability condition, gives that $F$ is Borda; this is established in Theorem \ref{measurable}.

Measurability for an anonymous SWF is defined in \autocite{paper1} as follows:

\begin{defn}[Measurability for anonymous SWFs]\label{def:measurable}
An anonymous SWF $F = F' \circ \tau$ is said to be measurable if $F'$ is a measurable function from $\set{\vec{\varepsilon} \in \mathbbm{R}_{\geq 0}^{\mathfrak{R}} | \sum_r \varepsilon_r = 1}$, equipped with the Lebesgue measure, onto $\mathfrak{R}_=$ equipped with the discrete measure.
\end{defn}

Measurability is implied by either P or PR, so we get Corollary \ref{IUAkMaskin}, which states that requiring either ``increasingness'' condition leaves only the unweighted Borda rules (the tie rule in the case of PR, and the positive Borda rule in the case of either P or PR). On the other hand, without the measurability condition, pathological counterexamples exist; we give these in Theorem \ref{IUAkchoice}.

In section \ref{FinUAk} we address the case of a finite set of voters $V$ with $|V|=n$. Here, there are no complications around measurability and the axiom of choice, and we simply prove Theorem \ref{FUAk}, which states that given U, MIIA, A and N, the SWF must be Borda.

\section{The unrestricted domain for more than three candidates and infinite voters}\label{InfUAk}

After identifying non-Borda elections on the unrestricted domain for three candidates in \autocite{paper1}, we saw that restricting to the Condorcet Cycle domain forces the election to be Borda (subject to measurability conditions), as the elections on which the result disagrees with the Borda rule are excluded from the domain of the SWF. Another way in which we can attempt to force SWFs to be the Borda rule is by \emph{increasing} the domain of the election, such that non-Borda SWFs on the smaller domain cannot be extended to the larger domain without breaking other conditions such as MIIA. It turns out that this occurs as soon as we add any number of candidates to the election.

\begin{thm}\label{IUAk}
Let $F$ be a SWF with voters $V$ and candidates $C$, satisfying U, MIIA, A and N. If $V$ is an infinite set and $|C|>3$ then $F$ is Borda-determined.
\end{thm}

\begin{proof}

Given the SWF $F$, we apply Lemma \ref{namiia} to derive a unique relative weight-based SWF $f'$. Lemma \ref{Ftofprime} tells us that this function is consistent. The proof proceeds by using transitive election tables to evaluate $f'$ on different parts of its domain, until we determine that it can indeed be written as $f' = g \circ d_1$ for some function $g: \mathbbm{R} \rightarrow \set{W,T,L}$.

We begin in the case of $k=4$. We denote by $A'_4$ the domain of relative weight-based SWFs on four candidates (which is $\set{\vec{\alpha} \in \mathbbm{R}_{\geq 0}^{D_4}| \sum_{i \in D_4} \alpha_i = 1}$).

\begin{lem}\label{2candsym}
For $f': A'_4 \rightarrow \set{W,T,L}$ consistent, and any $\alpha_3+\alpha_2+\alpha_1 = 1/2$ with $\alpha_i \geq 0$, we have $f'(\alpha_3, \alpha_2, \alpha_1, \alpha_1, \alpha_2, \alpha_3) = T$.
\end{lem}

\begin{proof}
Let $f'(\vec{\alpha}) = X$ and consider the following tying election table:

\begin{equation}\begin{aligned}
\left(\begin{array}{c c | c c | c c }
(\alpha_1) & & (\alpha_2) & & \alpha_3 & \\ \hline
1 & -1 & 2 & -2 & 3 & -3 \\
-1 & 1 & -2 & 2 & -3 & 3 
\end{array}\right)
\end{aligned}\end{equation}

Then by Lemma \ref{tyingelectiontable} we have $X=T$ as required.
\end{proof}

\begin{lem}\label{firstcomb}
Let $\vec{\alpha} \in A'_4$ have $d_1(\vec{\alpha}) = 0$ and satisfy the following inequalities:

\begin{equation}\label{eq:firstcombconditions}
\begin{matrix}
\alpha_3 & \leq & \alpha_{-2}+\alpha_{-3} & \leq & \alpha_1 + \alpha_2 + \alpha_3 \\
\alpha_{-3} & \leq & \alpha_2+\alpha_3 & \leq & \alpha_{-1} + \alpha_{-2} + \alpha_{-3}
\end{matrix}
\end{equation}

Then $f(\vec{\alpha}) = T$.
\end{lem}

\begin{proof}
Let $\vec{\alpha}$ be such a vector. We let $s = \max(0, \alpha_{-3}-\alpha_3)$ and $t = \max(0, \frac{\alpha_{-2}-\alpha_2 + \alpha_{-3}-\alpha_3}{2})$. Then we consider the following tying election table (which we transpose due to width constraints):

\begin{equation}\label{eq:firstcombTET}\begin{aligned}
\left(\begin{array}{c | c c }
(\frac{\alpha_2 - s}{2}) & -1 & 3 \\ 
& 3 & -1 \\ \hline
(t) & -2 & 3 \\
& 3 & -2 \\ \hline
(\frac{\alpha_3}{2}) & 1 & 2 \\
& 2 & 1 \\ \hline
(\frac{\alpha_1}{2} - t) & -1 & 2 \\
& 2 & -1 \\ \hline
(t + \frac{\alpha_2 - \alpha_{-2} + \alpha_3 - \alpha_{-3}}{2}) & -3 & 2 \\
& 2 & -3 \\ \hline
(s) & 1 & 1 \\ \hline
(\frac{\alpha_1 + \alpha_2 - \alpha_{-2}}{2} + \alpha_3 - \alpha_{-3} - t) & -2 & 1 \\
& 1 & -2 \\ \hline
(\frac{\alpha_{-2} + \alpha_{-3}-\alpha_3 - s}{2}) & -3 & 1 \\
& 1 & -3 \\ \hline
(s + \alpha_3 - \alpha_{-3}) & -1 & -1 \\ \hline
(\alpha_{-3}/2) & -2 & -1 \\
& -1 & -2
\end{array}\right)
\end{aligned}\end{equation}

The sum of all the $w_i$ values is $4\alpha_3 + 3\alpha_2 + 2\alpha_1 - \alpha_{-2} - 2 \alpha_{-3} = d_1(\vec{\alpha}) + \norm{\vec{\alpha}}_1 = 1$, so this is a valid election as long as all $w_i$ are non-negative. It suffices to check that

\begin{enumerate}
\item $t \leq \alpha_1 /2$,
\item $t \leq \frac{\alpha_1 + \alpha_2 - \alpha_{-2}}{2} + \alpha_3 - \alpha_{-3}$, and
\item $s \leq \alpha_{-2} + \alpha_{-3}-\alpha_3$
\end{enumerate}

All the other values are trivially non-negative.

For the first inequality, if $t=0$ then this is immediate. Otherwise $t = \frac{\alpha_{-2}-\alpha_2 + \alpha_{-3}-\alpha_3}{2}$ and the inequality becomes $\alpha_{-2} + \alpha_{-3} \leq \alpha_1+\alpha_2+\alpha_3$, which is given as an assumption in (\ref{eq:firstcombconditions}).

For the second inequality, if $t=0$ then this becomes $\alpha_{-2} + 2\alpha_{-3} \leq \alpha_1 + \alpha_2 + 2\alpha_3$. Now we are given in (\ref{eq:firstcombconditions}) that $\alpha_2 + \alpha_3 \leq \alpha_{-1}+\alpha_{-2}+\alpha_{-3}$; this rearranges to $\alpha_1 - \alpha_{-1} + \alpha_2 - \alpha_{-2} + \alpha_3 - \alpha_{-3} \leq \alpha_1$. Subtracting the (vanishing) expression for $d_1(\vec{\alpha})$ from the left hand side gives $\alpha_{-2} - \alpha_2 + 2(\alpha_{-3}-\alpha_3) \leq \alpha_1$, which rearranges to the required inequality.

On the other hand, if $t = \frac{\alpha_{-2}-\alpha_2 + \alpha_{-3}-\alpha_3}{2}$ then the inequality becomes $0 \leq \alpha_1/2 + \alpha_2 - \alpha_{-2} + 3(\alpha_3 - \alpha_{-3})/2$. Subtracting half of the expression for $d_1(\vec{\alpha})$ from the right hand side gives $0 \leq \alpha_{-1}/2$ which is immediate.

For the third inequality, we are given in (\ref{eq:firstcombconditions}) that $\alpha_3 \leq \alpha_{-2} + \alpha_{-3}$ so if $s=0$ the inequality holds. Otherwise $s = \alpha_{-3} - \alpha_3$ and the inequality becomes $\alpha_{-2} \geq 0$ which is immediate.

Therefore $w$ is a valid weight vector for the table, and $M_1$, $M_2$ and $M_1+M_2$ are all vectors in $D_4^t$. Moreover

\begin{equation}
\vec{\alpha}(w,M_1) = \vec{\alpha}(w,M_2) = \begin{pmatrix}
\frac{\alpha_2 - s}{2} + t \\
\frac{\alpha_2 - \alpha_{-2} + 2\alpha_3 - \alpha_{-3} + \alpha_1}{2} \\
\frac{\alpha_2 +s + \alpha_1 - \alpha_{-3}}{2} -t + \alpha_3 \\
\frac{\alpha_2 +s + \alpha_1 - \alpha_{-3}}{2} -t + \alpha_3 \\
\frac{\alpha_2 - \alpha_{-2} + 2\alpha_3 - \alpha_{-3} + \alpha_1}{2} \\
\frac{\alpha_2 - s}{2} + t \end{pmatrix}
\end{equation}

which fulfils the conditions for Lemma \ref{2candsym} to hold, so $f'(\vec{\alpha}(w,M_1)) = f'(\vec{\alpha}(w,M_2)) = T$. Therefore $(w,M)$ is a $T$-transitive election table, and $f'(\vec{\alpha}(w,M_1+M_2)) = T$. 

But computing $\vec{\alpha}(w,M_1+M_2)$ gives precisely $\vec{\alpha}$, so $f(\vec{\alpha}) = T$ as required.
\end{proof}

We apply the same method again to give a broader set of tying vectors.

\begin{lem} \label{secondcomb}
Let $\vec{\alpha} \in A'_4$ have $d_1(\vec{\alpha}) = 0$ and satisfy the following inequalities:

\begin{equation} \label{eq:secondcombconditions}
\begin{aligned}
\alpha_{-3} \leq 2\alpha_2 + 2\alpha_3 \\
\alpha_{3} \leq 2\alpha_{-2} + 2\alpha_{-3}
\end{aligned}
\end{equation}

Then $f(\vec{\alpha}) = T$.
\end{lem}

\begin{proof}
Let $\vec{\alpha}$ be such a vector. Then we consider the following tying election table:

\begin{equation} \label{eq:secondcombTET}\begin{aligned}
\left(\begin{array}{c | c c }
(\frac{\alpha_1}{2}) & -2 & 3 \\ 
& 3 & -2 \\ \hline
(\frac{\alpha_3}{2}) & 1 & 2 \\
& 2 & 1 \\ \hline
(\frac{\alpha_{-1}}{2}) & -3 & 2 \\
& 2 & -3 \\ \hline
(\alpha_2) & 1 & 1 \\ \hline
(\alpha_{-2}) & -1 & -1 \\ \hline
(\frac{\alpha_{-3}}{2}) & -2 & -1 \\
& -1 & -2 \\
\end{array}\right)
\end{aligned}\end{equation}

The $w_i$ values are all non-negative and sum to $1$. Therefore $w$ is a valid weight vector for the table, and $M_1$, $M_2$ and $M_1+M_2$ are all vectors in $D_4^t$. Moreover

\begin{equation}\label{eq:secondcombalpha}
\vec{\alpha}(w,M_1) =\vec{\alpha}(w,M_2) = \begin{pmatrix}
\frac{\alpha_1}{2} \\
\frac{\alpha_3 + \alpha_{-1}}{2} \\
\frac{\alpha_3}{2} + \alpha_2 \\
\frac{\alpha_{-3}}{2} + \alpha_{-2} \\
\frac{\alpha_1 + \alpha_{-3}}{2} \\
\frac{\alpha_{-1}}{2} \end{pmatrix}
\end{equation}

Now we will confirm that this vector satisfies the conditions in Lemma \ref{firstcomb}. Calculating $d_1$ on the vector gives $\frac{3\alpha_3 + 2\alpha_2 + \alpha_1 -\alpha_{-1} - \alpha_{-2} - 3\alpha_{-3}}{2}$ which is $d_1(\vec{\alpha})/2 = 0$ by assumption. As for the four inequalities given in (\ref{eq:firstcombconditions}), the left-hand inequalities are immediate, while the ones on the right become $\alpha_3 + \alpha_1 + \alpha_{-1} \leq \alpha_1 + \alpha_{-1} + 2\alpha_{-2} + 2\alpha_{-3}$ and $\alpha_{-3} + \alpha_{-1} + \alpha_1 \leq \alpha_{-1} + \alpha_1 + 2\alpha_2 + 2\alpha_3$.

These simplify to the two inequalities we took as assumptions in (\ref{eq:secondcombconditions}), so we can indeed apply Lemma \ref{firstcomb} and get that $f'(\vec{\alpha}(w,M_1))=f'(\vec{\alpha}(w,M_2))=T$. Therefore $(w,M)$ is a $T$-transitive election table, and $f'(\vec{\alpha}(w,M_1+M_2)) = T$. 

But computing $\vec{\alpha}(w,M_1+M_2)$ gives precisely $\vec{\alpha}$, so $f(\vec{\alpha}) = T$ as required.
\end{proof}

A third and final application of this method gives all $\vec{\alpha} \in A'_4$ for which the Borda rule gives a tie.

\begin{lem} \label{thirdcomb}
Let $\vec{\alpha} \in A'_4$ have $d_1(\vec{\alpha}) = 0$. Then $f(\vec{\alpha}) = T$.
\end{lem}

\begin{proof}
We use the same transitive election table as in (\ref{eq:secondcombTET}). This time we need to show that the value of $\vec{\alpha}(w,M_1)$ given in (\ref{eq:secondcombalpha}) satisfies the assumptions of (\ref{eq:secondcombconditions}). The two inequalities become

\begin{equation}
\begin{aligned}
\frac{\alpha_1}{2} &\leq \alpha_1 + \alpha_{-1} + \alpha_{-3} \\
\frac{\alpha_{-1}}{2} &\leq \alpha_3 + \alpha_{1} + \alpha_{-1}
\end{aligned}
\end{equation}

Both of these are immediate. Once again, this means that $(w,M)$ is a $T$-transitive election table, and $f'(\vec{\alpha}) = f'(\vec{\alpha}(w,M_1+M_2)) = T$, as required.
\end{proof}

This means that for any SWF $F$ on four candidates with conditions U, MIIA, A and N, the relative SWF $f'$ gives a tie whenever $\phi \circ d_1$ the Borda rule gives a tie. What can be said about all other $\vec{\alpha} \in A$?

For $b$ a real number in the interval $[0,1/10]$ we define $\vec{\alpha}_b = \left(0,1/4,1/4+b/2, 1/4 - b/2,1/4,0 \right)$. Note that $d_1(\vec{\alpha}_b)=b$. We define $f'(b) = f'(\vec{\alpha}_b)$.

\begin{lem}\label{firstgeq}
For $\vec{\alpha} \in A'_4$ with $\alpha_3 = \alpha_{-3} = 0$, $\alpha_1 + \alpha_{-1} \geq 1/2$ and $0 \leq b = d_1(\vec{\alpha}) \leq 1/10$, we have $f'(\vec{\alpha}) = f'(b)$.
\end{lem}

\begin{proof}

Let $\vec{\alpha}$ be such a vector. We let $c = \min(\alpha_1, 1/4 - b/2, \alpha_1 + \alpha_{-1}-1/2) \geq 0$ and we write out the following transitive election table:

\begin{equation} \label{eq:firstgeqTET}\begin{aligned}
\left(\begin{array}{c | c c }
(c) & -1 & 2 \\ \hline
(\alpha_1+\alpha_{-1} - 1/2 - c) & 1 & -2 \\ \hline
(\frac{\alpha_1-c}{2}) & 2 & -1 \\ 
& -2 & 3 \\ \hline
(\frac{1/2 - \alpha_1 + c}{2}) & 2 & -3 \\
& -2 & 1 \\ \hline
(\frac{\alpha_2}{\alpha_2 + \alpha_{-2}} \cdot (3/4 + b/2 - \alpha_1 - \alpha_{-1} +c)) & 1 & 1 \\ \hline
(\frac{\alpha_{-2}}{\alpha_2 + \alpha_{-2}} \cdot (3/4 + b/2 - \alpha_1 - \alpha_{-1} +c)) & 1 & -3 \\ \hline
(\frac{\alpha_2}{\alpha_2 + \alpha_{-2}} \cdot (1/4 - b/2 - c)) & -1 & 3 \\ \hline
(\frac{\alpha_{-2}}{\alpha_2 + \alpha_{-2}} \cdot (1/4 - b/2 - c)) & -1 & -1
\end{array}\right)
\end{aligned}\end{equation}

We check that the $1$-norm of $w$ is $1$. We also need to show that each $w_i$ is non-negative. Some cases are immediate and the others reduce to showing that $c \geq \alpha_1 - 1/2$ and $c \geq \alpha_1 + \alpha_{-1} - 3/4 - b/2$.

If $c < \alpha_1 - 1/2$ then $c = 1/4-b/2$ and we have $3/4 - b/2 < \alpha_1$. But then $b = d_1(\vec{\alpha}) > 3/4 - b/2 - 2(1/4+b/2)$ and $b > 1/10$, a contradiction.

If $c < \alpha_1 + \alpha_{-1} - 3/4 - b/2$ then either $3/4 + b/2 < \alpha_{-1}$ which immediately means that $d_1(\vec{\alpha})<0<b$, a contradiction, or $1 < \alpha_1 + \alpha_{-1}$, a contradiction. So indeed all of the inequalities hold and $w$ represents a real partition of the voters.

Now $\vec{\alpha}(w,M_1) = \vec{\alpha}_b$ and $\vec{\alpha}(w, M_1+M_2) = \vec{\alpha}$. On the other hand, $d_1(\vec{\alpha}(w, M_2))=0$ so by Lemma \ref{thirdcomb} we know that $f'(\vec{\alpha}(w,M_2))=T$, so for consistency we must have $f'(b) = f'(\vec{\alpha}(w,M_1)) = f'(\vec{\alpha}(w,M_1+M_2)) = f'(\vec{\alpha})$ as required.
\end{proof}

\begin{lem}\label{secondgeq}
For $\vec{\alpha} \in A'_4$ with $\alpha_3 = \alpha_{-3} = 0$ and $0 \leq b = d_1(\vec{\alpha}) \leq 1/10$, we have $f'(\vec{\alpha}) = f'(b)$.
\end{lem}

\begin{proof}
Let $\vec{\alpha}$ be a such a vector. If $\alpha_1+\alpha_{-1} \geq 1/2$ then we are done by Lemma \ref{firstgeq}. So instead suppose that $\alpha_1 + \alpha_{-1} < 1/2$. Set $\lambda = \frac{1}{2} + \frac{b}{2 + 2\alpha_1 + 2\alpha_{-1}}$. and note that $0 \leq \lambda \leq 1$. Now consider the following transitive election table:

\begin{equation} \label{eq:secondgeqTET}\begin{aligned}
\left(\begin{array}{c | c c }
(\lambda \alpha_1) & 2 & -1 \\ \hline
(\lambda \alpha_{-1}) & 2 & -3 \\ \hline
((1-\lambda) \alpha_1) & -2 & 3 \\ \hline
((1-\lambda) \alpha_{-1}) & -2 & 1 \\ \hline
(\lambda \alpha_2) & 1 & 1 \\ \hline
(\lambda \alpha_{-2}) & 1 & -3 \\ \hline
((1-\lambda) \alpha_2) & -1 & 3 \\ \hline
((1-\lambda) \alpha_{-2}) & -1 & -1
\end{array}\right)
\end{aligned}\end{equation}

Trivially these values of $w_i$ are non-negative and they add up to $1$, so this is a valid $w$ vector. We have $\vec{\alpha}(w, M_1) = (\lambda (\alpha_1+\alpha_{-1}), \lambda(\alpha_2+\alpha_{-2}), (1-\lambda)(\alpha_2+\alpha_{-2}), (1-\lambda)(\alpha_1+\alpha_{-1}))$ which satisfies the conditions of Lemma \ref{firstgeq}, so $f'(\vec{\alpha}(w, M_1)) = f'(b)$. On the other hand $d_1(\vec{\alpha}(w,M_2)) = 0$, so $f'(\vec{\alpha}(w,M_2)) = T$; therefore $\vec{\alpha}(w,M_1+M_2) = \vec{\alpha}$ has $f'(\vec{\alpha}) = f'(b)$ as required.
\end{proof}

We now know that within a small region of the space of pairwise comparisons between candidates, the pairwise ranking is determined only by the Borda score. We require one more step to extend this to the entire space.

\begin{lem}\label{doubling}
For any $b \in [-3,3]$ and for any $\vec{\alpha} \in A'_4$ with $d_1(\vec{\alpha}) = b$, there exists $\vec{\alpha}'$ with $d_1(\vec{\alpha}') = b/2$, $\alpha'_3 = \alpha'_{-3} = 0$ and $f(\vec{\alpha}') = f(\vec{\alpha})$.
\end{lem}

\begin{proof}
Given $\vec{\alpha}$ we consider the following transitive election table:

\begin{equation} \label{eq:doublingTET}\begin{aligned}
\left(\begin{array}{c | c c }
(\alpha_3/2) & 1 & 2 \\ 
& 2 & 1 \\ \hline
(\alpha_2) & 1 & 1 \\ \hline
(\alpha_1/2) & -1 & 2 \\
& 2 & -1 \\ \hline
(\alpha_{-1}/2) & -2 & 1 \\
& 1 & -2 \\ \hline
(\alpha_{-2}) & -1 & -1 \\ \hline
(\alpha_{-3}/2) & -1 & -2 \\
& -2 & -1
\end{array}\right)
\end{aligned}\end{equation}

Now $\vec{\alpha}(w,M_1+M_2) = \vec{\alpha}$ and

\begin{equation}
\begin{aligned}
\vec{\alpha}(w,M_1) =\vec{\alpha}(w,M_2) = (0, \frac{\alpha_3 + \alpha_1}{2}, \frac{\alpha_3 + 2\alpha_2 + \alpha_{-1}}{2}, \frac{\alpha_1 + 2\alpha_{-2} + \alpha_{-3}}{2}, \frac{\alpha_{-1} + \alpha_{-3}}{2}, 0)
\end{aligned}
\end{equation}

We call this vector $\vec{\alpha}'$. Now let $X = f'(\vec{\alpha}')$; then (\ref{eq:doublingTET}) is an $X$-transitive election table, so $f'(\vec{\alpha}) = X = f(\vec{\alpha}')$ as required.
\end{proof}

\begin{lem}\label{layers}
For any $\vec{\alpha} \in A'_4$, $f(\vec{\alpha})$ is determined by $d_1(\vec{\alpha})$.
\end{lem}

\begin{proof}
We begin by taking $\vec{\alpha}$ with $0 \leq b = d_1(\vec{\alpha}) \leq 1/5$. Now by Lemma \ref{doubling} we can construct $\vec{\alpha}'$ with $0 \leq b/2 = d_1(\vec{\alpha}') \leq 1/10$ such that $f'(\vec{\alpha}) = f'(\vec{\alpha}')$. But $f'(\vec{\alpha}') = f'(b/2)$ by Lemma \ref{secondgeq}, so $f'(\vec{\alpha}) = f'(b/2)$ is determined only by $b = d_1(\vec{\alpha})$.

Now for any $\vec{\alpha}$ with $b = d_1(\vec{\alpha})\geq 0$ we can define a chain of $\vec{\alpha} = \vec{\alpha}^{(0)}, \vec{\alpha}^{(1)}, \dots, \vec{\alpha}^{(4)}$ by the construction in (\ref{eq:doublingTET}). Again, $f'(\vec{\alpha}) = f'(\vec{\alpha}^{(4)})$, but $d_1(\vec{\alpha}^{(4)}) = d_1(\vec{\alpha})/16 \leq 1/5$, so $f'(\vec{\alpha}) = f'(b/32)$ and is determined only by its Borda score.

Finally, for $d_1(\vec{\alpha})<0$, we have $d_1(\varrho(\vec{\alpha})) > 0$ so $f'(\varrho(\vec{\alpha})) = f'(-b)$. But since $f'$ is consistent we have $f'(\varrho(\vec{\alpha})) = -f'(\vec{\alpha}) = -f'(-b)$, so once again $f'(\vec{\alpha})$ is determined only by its Borda score.
\end{proof}

We now generalise to the case of $k$ candidates, completing the proof of Theorem \ref{IUAk}.

\begin{lem}\label{layersk}
For any $k \geq 4$ and $\vec{\alpha} \in A'_k$, the value of $f(\vec{\alpha})$ is determined by $d_1(\vec{\alpha})$.
\end{lem}

\begin{proof}
We prove this inductively on $k$; we already have the case of $k=4$ from Lemma \ref{layers}. Now assuming that the statement hold for $4 \leq k \leq n$, consider the case of $k=n+1$. Note that if we restrict $E$ to the elections in which the candidate $c_{n+1}$ is ranked last by every voter, we must still have a consistent election between the candidates $(c_i)_{i=1}^n$ obeying all of the given conditions. Therefore we can apply the inductive hypothesis to see that for any $\vec{\alpha} \in A_n \subset A_{n+1}$, $f'(\vec{\alpha})$ is determined by $d_1(\vec{\alpha})$; we call this value $f'(b)$.

Now for any $\vec{\alpha}$, with $d_1(\vec{\alpha})=b$, consider the following transitive election table, generalising the one given in (\ref{eq:doublingTET}).

\begin{equation} \label{eq:layerskTET}\begin{aligned}
\left(\begin{array}{c | c c }
(\alpha_{2i})_{1 \leq i \leq \lfloor n/2\rfloor} & i & i \\ \hline
(\alpha_{2i+1}/2)_{1 \leq i \leq \lfloor (n-1)/2 \rfloor} & i & i+1 \\
& i+1 & i \\ \hline
(\alpha_1/2) & -1 & 2 \\
& 2 & -1 \\ \hline
(\alpha_{-1}/2) & -2 & 1 \\
& 1 & -2 \\ \hline
(\alpha_{-2i})_{1 \leq i \leq \lfloor n/2\rfloor} & -i & -i \\ \hline
(\alpha_{-2i-1}/2)_{1 \leq i \leq \lfloor (n-1)/2 \rfloor} & -i & -i-1 \\
& -i-1 & -i
\end{array}\right)
\end{aligned}\end{equation}

Now each $w_i \geq 0$ and $\sum_i w_i = 1$. Moreover, $\vec{\alpha}(w,M_1+M_2)= \vec{\alpha}$, whereas $\vec{\alpha}(w,M_1) = \vec{\alpha}(w,M_2) = \vec{\alpha}'$ is in $A_n$, with $d_1(\vec{\alpha}) = b/2$ and therefore $f'(\vec{\alpha}')=f'(b/2)$. Thus (\ref{eq:layerskTET}) is a $f'(b/2)$-transitive election table, and $f'(\vec{\alpha}) = f'(b/2)$. So $f'(\vec{\alpha})$ is indeed determined by $d_1(\vec{\alpha})$ for all of $A_{n+1}$, completing the inductive step. This completes the proof.
\end{proof}

Thus $f': A'_k \rightarrow \set{W,T,L}$ decomposes to $f' \circ d_1$, so $F$ is Borda-determined, proving Theorem \ref{IUAk}.
\end{proof}

We proceed to determine what kind of function $f': [1-k,k-1] \rightarrow \set{W,T,L}$ can be. We have already seen that $f'(b) = -f'(-b)$, corresponding to condition \ref{consistent:twoway} of the definition of consistency. The following is the corresponding result for condition \ref{consistent:threeway}.

\begin{lem}\label{consistentAk}
Let $F$ be a SWF with voters $V$ infinite and candidates $C$ with $|C|>3$, satisfying U, MIIA, A and N. Let $g: [1-k,k-1] \rightarrow \set{W,T,L}$ be such that $g \circ d_1$ is $F$'s pairwise weight-based SWF. Then for $b_1, b_2, b_3 \in [1-k,k-1]$, if $b_1+b_2+b_3 = 0$ then $\set{g(b_1),g(b_2),g(b_3)}$ is a consistent multiset.
\end{lem}

\begin{proof}

By Lemma \ref{Ftofprime}, $g \circ d_1$ must be consistent, so it suffices to show that for $b_1,b_2,b_3 \in [-1,1]$ with $b_1 + b_2 + b_3 = 0$, there exists a weight-based election $\vec{\varepsilon} \in E'_k$ for which

\begin{equation}
\begin{aligned}
d_1(\pi_{1,2}(\vec{\varepsilon})) &= b_1 \\
d_1(\pi_{2,3}(\vec{\varepsilon})) &= b_2 \\
d_1(\pi_{3,1}(\vec{\varepsilon})) &= b_3
\end{aligned}
\end{equation}

Then, for $b_1,b_2,b_3 \in [1-k,k-1]$, we let $b'_i = b_i/2^k \in [-1,1]$. By Lemma \ref{doubling} we know that $g(b'_i)=g(b_i)$, so this will complete the proof.

Either two of the $g(b_i)$ are equal, in which case assume without loss of generality that these are $b_1$ and $b_2$, or one of them takes each of the values $W$, $T$ and $L$, in which case assume without loss of generality that $g(b_1)=T$ and set $X = g(b_2)$.

Now the possible values of $(b_1,b_2)$ are all convex combinations of

\begin{align*}
P = \set{(1,0) , (0,1) , (1,-1) , (-1,1) , (-1,0) , (0,-1)}
\end{align*}

Let $(b_1,b_2) = \sum_{(x,y) \in P} \lambda_{x,y} (x,y)$, with $\lambda_{x,y} \geq 0$ and $\sum_{(x,y) \in P} \lambda_{x,y} = 1$. Then the following is an $X$-transitive election table:

\begin{equation}
\begin{aligned}
\left(\begin{array}{c | c c }
(\lambda_{1,0}/2) & 1 & 2 \\
& 1 & -2 \\ \hline
(\lambda_{0,1}/2) & 2 & 1 \\
& -2 & 1 \\ \hline
(\lambda_{1,-1}/2) & 1 & 1 \\
& 1 & -3 \\ \hline
(\lambda_{-1,1}/2) & 1 & 1 \\
& -3 & 1 \\ \hline
(\lambda_{-1,0}/2) & -1 & 2 \\
& -1 & -2 \\ \hline
(\lambda_{0,-1}/2) & -1 & 2 \\
& -1 & -2
\end{array}\right)
\end{aligned}
\end{equation}

because $d_1(\vec{\alpha}(w,M_1)) = b_1$ and $d_1(\vec{\alpha}(w,M_2)) = b_2$. Therefore $g(b_1+b_2) = f'(\vec{\alpha}(w,M_1+M_2))=X$, so $g(b_3) = -X$. Therefore $\set{g(b_1),g(b_2),g(b_3)} = \set{X,X,-X}$ or $\set{T,X,-X}$ is consistent, as required.
\end{proof}

We now recall the following result from \autocite{paper1}:

\begin{lem}[Lemma 4.6 in \autocite{paper1}]\label{measurable}
Let $-a < 0 < b$ with $2a \geq b$ and $2b \geq a$, and let $f'$ be a measurable function from $[-a,b]$ to $\set{W,T,L}$. Suppose that for any $x_1,x_2,x_3 \in [-a,b]$ with $x_1+x_2+x_3=0$, the multiset $\set{f'(x_i)}_i$ is consistent. Then $f'(x) = \varphi(\kappa x)$ for some real number $\kappa$.
\end{lem}

Lemma \ref{consistentAk} gives us precisely the conditions we need to apply Lemma \ref{measurable}, with $a = b = k-1$. This tells us that $f'$ is constant at some $X \in \set{W,T,L}$ for $d_1(\vec{\alpha})<0$, constant at $-X$ for $d_1(\vec{\alpha})>0$ and (as we already knew) $f'(0)=T$. Thus $f'$ exactly matches one of the three unweighted Borda rules, so $F$ is a Borda rule as required.

\begin{cor}\label{IUAkmeasurable}
Let $F$ be a SWF with voters $V$ and candidates $C$, satisfying U, MIIA, A and N, with $F'$ measurable. If $V$ is an infinite set and $|C|>3$, then $F$ is a Borda rule.
\end{cor}

Positive Responsiveness is enough to guarantee measurability, as we observed in the proof of Corollary 4.8 in \autocite{paper1}. Unlike in that case, the Pareto principle is also sufficient to guarantee the positive unweighted Borda rule. For any $b \in [1,2]$ we consider $\vec{\alpha}^{(b)}$ with $\alpha_2 = b-1$, $\alpha_1 = 2-b$ and $\alpha_i = 0$ elsewhere. Then $\vec{\alpha}^{(b)}$ satisfies the criterion of the Pareto condition, so $f'(b) = f'(\vec{\alpha}^{(b)}) = W$ for $b \in [1,2]$. But for any $0<b$ there is some $m \in \mathbbm{Z}$ for which $1 \leq 2^m b \leq 2$, and $f'(b) = f'(2^m b) =W$, and $f' = \phi \circ d_1$ as required. This completes the proof of Maskin's Theorem for $k>3$ candidates.

\begin{cor}[Maskin's Theorem on $k>3$ candidates]\label{IUAkMaskin}
Let $F$ be a SWF with voters $V$ and candidates $C$, satisfying U, MIIA, A and N. If $V$ is an infinite set and $|C|>3$, then:
\begin{enumerate}
\item if $F$ satisfies PR then it is the positive unweighted Borda rule or the tie rule.
\item if $F$ satisfies P then it is the positive unweighted Borda rule.
\end{enumerate}
\end{cor}

On the other hand, when we drop the measurability condition and assume the axiom of choice we find pathological alternatives to the Borda rule. In some sense, without measurability there is no way for an election rule to ``know'' whether it is partitioning Borda scores based on whether they are positive or negative - consistency only requires that the $W$ and $L$ sets are complementary (aside from $0$), closed under addition and map to one another under $x \mapsto -x$, and we have constructed such sets in \autocite{paper1}:

\begin{lem}[Lemma 4.10 in \autocite{paper1}]\label{choiceconstruction}
There exists a subset $M \subset \mathbbm{R}^*$ of the non-zero reals such that
\begin{enumerate}
\item \label{nonborda} $1,-\sqrt{2} \in M$
\item \label{rationaldilation} If $t \in M$ then $qt \in M$ for all positive rational numbers $q$
\item \label{closed} If $s,t \in M$ then $s+t \in M$
\item \label{partition} $M$ and $-M$ partition the set $\mathbbm{R}^*$
\end{enumerate}
\end{lem}

We are now ready to construct non-Borda SWFs:

\begin{thm}\label{IUAkchoice}
Assuming the axiom of choice, there exists a strongly non-Borda SWF on $V$ infinite and $|C| > 3$ satisfying U, MIIA, A and N.
\end{thm}

\begin{proof}
We take the construction of $M$ given in Lemma \ref{choiceconstruction}. We define $\varphi_M: [1-k,k-1] \rightarrow \set{W,T,L}$ as follows:

\begin{equation}
\varphi_M(d) = \begin{cases}
W & d \in M \\
T & d=0 \\
L & d \in -M
\end{cases}
\end{equation}

We let our SWF $F$ be Borda-determined with $f' = \varphi_M \circ d_1$. It remains to check that $f'$ is strongly non-Borda and consistent.

The relative election $\vec{\alpha}$ given by $\alpha_1 = 1$ and $\alpha_i=0$ elsewhere has $d_1(\vec{\alpha}) = 1 \in M$, so $f'(\vec{\alpha}) = W$. Therefore, for $\pi_{i,j}(\vec{\varepsilon}) = \vec{\alpha}$, $F$ ranks $c_i$ above $c_j$ but $B_{-1}$ ranks $c_j$ above $c_i$. So $F$ is not the negative unweighted Borda rule, or weakly Borda with respect to it, or the tie rule. On the other hand, the relative election $\vec{\alpha}$ given by $\alpha_1 = 2-\sqrt{2}$, $\alpha_2 = \sqrt{2}-1$ and $\alpha_i = 0$ elsewhere has $d_1(\vec{\alpha}) = \sqrt{2} \in -M$, so $f'(\vec{\alpha}) = L$.  Therefore, for $\pi_{i,j}(\vec{\varepsilon}) = \vec{\alpha}$, $F$ ranks $c_j$ above $c_i$ but $B_1$ ranks $c_i$ above $c_j$. So $F$ is not the positive unweighted Borda rule, or weakly Borda with respect to it. Finally, we recall the following result from \autocite{paper1}:

\begin{lem}[Lemma 2.25 in \autocite{paper1}]\label{weightinglemma}
For an election domain $E$ containing a ballot domain with ballot $\mathfrak{B}$, if a SWF $F$ fulfils condition A and is weakly Borda with respect to a Borda rule $B_w$ then it is weakly Borda with respect to an unweighted Borda rule $B$.
\end{lem}

The definition of a ballot domain can be found in \autocite{paper1}; all that we need to know here is that the unrestricted domain is a ballot domain. Lemma \ref{weightinglemma} allows us to conclude that $F$ is not weakly Borda with respect to any weighted Borda rule, so in fact it is strongly non-Borda.

For any election $e$ and for any three candidates $c_{i_1}, c_{i_2}, c_{i_3}$ we let $\tau_A \circ \pi_{i_j,i_{j+1}}(\vec{\varepsilon} = \vec{\alpha}_j$. Then $b_j = d_1(\vec{\alpha}_j) = b_1(c_{i_j},e) - b_1(c_{i_{j+1}},e)$, so $b_1+b_2+b_3 = 0$. Therefore $\set{f'(\vec{\alpha}_j)}_j = \set{\varphi_M(b_j)}_j$. This multiset can only be consistent. If any $b_j=0$ then $\varphi_M(b_j)=T$, and $b_{j+1} = -b_{j+2}$ so either $b_{j+1}=0$ and $\set{\varphi_M(b_j)}_j=\set{T,T,T}$, or $b_{j+1}$ and $b_{j+2}$ are in $M$ and $-M$ in some order and $\set{\varphi_M(b_j)}_j \set{W,T,L}$. On the other hand, if no $b_j=0$ then some $b_j$ and $b_{j+1}$ map to the same $X \in W,L$; in other words, both values are in $M$, or both are in $-M$. In either case, $b_j+b_{j+1}$ is also in that set by condition \ref{closed} from Lemma \ref{choiceconstruction}, so $b_{j+2}$ is in the opposite set, and we get either $\set{\varphi_M(b_j)}_j = \set{W,W,L}$ or $\set{\varphi_M(b_j)}_j = \set{W,L,L}$.

Hence $f'$ is consistent and we have constructed the required strongly non-Borda SWF.
\end{proof}

In this section, we have extended the domain enough to force all elections to be SWFs, but we have relied heavily on the infinite set of voters (insofar as we have subdivided sets of voters into sets of any measure of our choosing).

\section{The unrestricted domain for more than three candidates and finite voters}\label{FinUAk}

In the finite case, we are forced to make different arguments due to limits on the divisibility of $n$. For example, Lemma \ref{2candsym} is meaningless if $n$ is odd. Nonetheless, we continue to use the method of iteratively building up the set of tying elections to the full set of Borda ties, using transitive election tables. In this case, for each $n$ the sets $E'_k = E'_{k,n}$ and $A'_k = A'_{k,n}$ are slightly different depending on $n$, but we normally suppress this notation, as it will be clear from context.

We will also define $B'_k \subset A'_k$ to be the vectors $\vec{\alpha} \in A'_k$ for which $d_1(\vec{\alpha})=0$.

Again, we know that a SWF $F$ satisfying conditions A, MIIA and N corresponds to a consistent function $f': A'_k \rightarrow \set{W,T,L}$, and our task is to prove that such an $f'$ must be of the form

\begin{equation}
f'(\vec{\alpha}) = \begin{cases}
X & d_1(\vec{\alpha}) > 0 \\
T & d_1(\vec{\alpha}) = 0 \\
-X & d_1(\vec{\alpha}) < 0
\end{cases}
\end{equation}

As in section \ref{InfUAk}, we begin with the case of $k=4$, and we suppress the $k$ notation for now, writing $A' = A'_4$ and so on. We begin by noticing some subsets of $A'$ on which the behaviour of $f'$ is familiar.

For $k \leq n/3$ a non-negative integer such that $2 \mid n-k$, we let $\vec{\alpha}_{(t,k)} = (0, t, \frac{n-k}{2}-t, \frac{n+3k}{2}-t, t-k, 0)$ for $k \leq t \leq \frac{n-k}{2}$.

\begin{lem}\label{diagonals}
For $2 \mid n-k$ and $k \leq n/3$, and for $k \leq t \leq \frac{n-k}{2}$, the vector $\vec{\alpha}_{(t,k)}$ has:
\begin{enumerate}
\item \label{diagonalsX} $f'(\vec{\alpha}_{(t,k)}) = f'(\vec{\alpha}_{(k,k)})$ for $t < \frac{n+3k}{6}$;
\item \label{diagonalsrX} $f'(\vec{\alpha}_{(t,k)}) = -f'(\vec{\alpha}_{(k,k)})$ for $t > \frac{n+3k}{6}$; and
\item \label{diagonalsT} $f'(\vec{\alpha}_{(t,k)}) = T$ for $t = \frac{n+3k}{6}$
\end{enumerate}
\end{lem}

To prove this, we will require the following result from \autocite{paper1}:

\begin{lem}\label{numberlineelection}
For $\ell \geq 0$ let $I = \set{0, 1, \dots, \ell}$ and suppose $g: I \rightarrow \set{W,T,L}$. Moreover, let $m \in \mathbbm{Z}$ be such that $\ell \leq m \leq 2\ell$, and suppose that for all $i, j, k$ with $i+j+k = m$, the multiset $\set{g(i), g(j), g(k)}$ is consistent. Then $g$ is given by $g(i) = \varphi(\kappa(i-m/3))$ for some $\kappa \in \mathbbm{R}$.
\end{lem}

\begin{proof}[Proof of Lemma \ref{diagonals}]
For all $0 \leq s \leq \frac{n-3k}{2}$ we let $g(s) = f'(\vec{\alpha}_{(s+k,k)})$. We will show that whenever $s_1 + s_2 + s_3 = \frac{n-3k}{2}$ (note that this is always an integer), the multiset $\set{g(s_1),g(s_2),g(s_3)}$ is consistent. Suppose without loss of generality that either $g(s_1)=g(s_2)=X$ or that $g(s_1)=X$ and $g(s_2)=T$. Then the following is an $X$-transitive election table:

\begin{equation}\begin{aligned}
\left(\begin{array}{c c c c c c}
(s_1+k) & (s_2+k) & (s_3+k) & (s_1) & (s_2) & (s_3) \\
\hline
2 & -1 & -1 & -2 & 1 & 1 \\
-1 & 2 & -1 & 1 & -2 & 1
\end{array}
\right)
\end{aligned}\end{equation}

because $f'(\vec{\alpha}(w,M_1)) = g(s_1)$ and $f'(\vec{\alpha}(w,M_2)) = g(t_2)$. Therefore $f'(\vec{\alpha}(w,M_1+M_2)) = X$, so $g(t_3) = f'(\varrho(\vec{\alpha}(w,M_1+M_2))) = -X$, giving the consistent multiset required.

Therefore we can apply Lemma \ref{numberlineelection}, setting $\ell = m = \frac{n-3k}{2}$, and conclude that $g$ is given by $g(i) = \varphi(\kappa(i-m/3))$ for some $\kappa \in \mathbbm{R}$. Then the corresponding values of $f'$ on $\vec{\alpha}_{(t,k)}$ are as described above, and the proof is complete.
\end{proof}

We now need to show that $X=T$ for each of these families of vectors; for brevity we say that $X_a = f'(\vec{\alpha}_{(a,a)})$ (recall that this is defined for $2 \mid n-a$). Since $f'(\varrho(\vec{\alpha})) = -f'(\vec{\alpha})$, this will also cover the families of vectors of the form $\vec{\alpha}_{t,-k} = (0,t-k,\frac{n+3k}{2}-t,\frac{n-k}{2}-t,t,0)$.

\begin{lem}\label{twelvesteptyingloop}
$X_a = T$ for $2 \mid n-a$ and $a \leq n/5$.
\end{lem}

\begin{proof}
Let $\ell = \frac{n-5a}{2}$ and note that $\vec{\alpha}_{(a,a)} = (0,a,a+\ell,3a+\ell,0,0)$ has $f'(\vec{\alpha}_{(a,a)})=X_a$. We will construct an $X_a$-transitive election table $(w,M,P)$ with $m=12$ and $t=7$, and $\vec{\alpha}(w,M_r) = \vec{\alpha}_{(a,a)}$ for each row $M_r$. We set:

\begin{equation*}
\begin{aligned}
M &= \left(\begin{array}{c c c c c | c c }
(a) & & & & & (\ell) & \\
\hline
-1 & -1 & -1 & +1 & +2 & -1 & +1 \\
-1 & -1 & +2 & +1 & -1 & -1 & +1 \\
-1 & -1 & +2 & -1 & +1 & -1 & +1 \\
-1 & +1 & -1 & -1 & +2 & +1 & -1 \\
+2 & +1 & -1 & -1 & -1 & +1 & -1 \\
-1 & -1 & +1 & +2 & -1 & -1 & +1 \\
+2 & -1 & -1 & +1 & -1 & +1 & -1 \\
-1 & -1 & -1 & +2 & +1 & -1 & +1 \\
-1 & +2 & -1 & -1 & +1 & -1 & +1 \\
-1 & +2 & +1 & -1 & -1 & +1 & -1 \\
+2 & -1 & +1 & -1 & -1 & +1 & -1 \\
+2 & +1 & -1 & -1 & -1 & +1 & -1
\end{array}\right)
 \\ \\
P &= (((((((x_1 x_2) x_3) (x_4 x_5)) x_6) x_7) (x_8 x_9)) ((x_{10} x_{11}) x_{12}))
\end{aligned}
\end{equation*}

It is easy to check column-wise that the partial sums given by matching parentheses in $P$ always give values in $D_4$. On the other hand, $\sum_{k=1}^{12} M_k = \vec{0}$ so by Lemma \ref{tyingelectiontable} we have $X_a=T$ as required.
\end{proof}

In particular, $X_0 = T$ for $n$ even and $X_1=T$ for $n \geq 5$ odd.

\begin{lem}\label{kleqn4}
$X_a = T$ for $a \leq n/4$ and $5 \leq n$ or $n$ even.
\end{lem}

\begin{proof}
Let $b \in \set{0,1}$ be chosen such that $2 \mid n-b$. Define $\vec{\beta} = \vec{\alpha}_{(\frac{n-a}{2},b)}$. We know from Lemma \ref{twelvesteptyingloop} that $f'(\vec{\beta}) = T$. Therefore the following is an $X_a$-transitive election table:

\begin{equation}\begin{aligned}
\left(\begin{array}{c c c c c c}
(\frac{a-b}{2}) & (\frac{n-a}{2}) & (\frac{a+b}{2}) & (\frac{n-3a}{2}) & (b) & (a-b) \\
\hline
2 & -1 & 2 & 1 & -1 & -1 \\
1 & 2 & -1 & -2 & -1 & -2
\end{array}\right)
\end{aligned}\end{equation}

This is because $f(\vec{\alpha}(w, M_1)) = f'(\vec{\alpha}_{(a,a)}) = X_a$ and $f(\vec{\alpha}(w, M_2)) = f'(\vec{\beta}) = T$. The fact that each column has non-negative integer multiplicity follows from the fact that $n$, $a$ and $b$ have the same parity and $b \leq 1$. Thus Lemma \ref{transitiveelectiontable} gives that $\vec{\alpha}(w,M_1+M_2) = (\frac{a-b}{2}, 0, \frac{n+b}{2}, \frac{n-3a}{2}, b, a-b) = \vec{\gamma}$ has $f'(\vec{\gamma}) = X_a$.

Therefore the following is also an $X_a$-transitive election table:

\begin{equation}\begin{aligned}
\left(\begin{array}{c c | c c | c c | c | c}
(a-b) & & (b) & & (\frac{a-3b}{2}) & & (\frac{n-2a+3b}{2}) & (\frac{n-4a+3b}{2}) \\
\hline
-3 & 1 & -2 & 3 & 3 & -1 & 1 & -1 \\
1 & -1 & 3 & -2 & -1 & 3 & 1 & -1
\end{array}\right)
\end{aligned}\end{equation}

In this instance, $\vec{\alpha}(w, M_1)=\vec{\alpha}(w, M_2) = \vec{\gamma}$. The multiplicities of each column are clearly integers as $n$, $a$ and $b$ have the same parity; they are non-negative since $a \leq n/4$ and as long as $a \geq 3b$; this is immediate for $a \geq 3$ and for $a=2$ since then $b=0$; for $a=1$, we already know that $X_1 = T$ since $n \geq 5$ so we can apply Lemma \ref{twelvesteptyingloop}.

Thus Lemma \ref{transitiveelectiontable} gives that $\vec{\alpha}(w,M_1+M_2) = (0, \frac{n-3b}{2}, 2b, 0, \frac{n-b}{2}, 0) = \vec{\delta}$ has $f'(\vec{\delta})=X_a$. But $\vec{\delta} = \vec{\alpha}_{(\frac{n-b}{2},-b)}$, so $X_a = f'(\vec{\delta}) = f'(\vec{\alpha}_{(\frac{n-b}{2},-b)}) = \varrho(f'(\vec{\alpha}_{(\frac{n-b}{2},b)})) = -T=T$. Hence $X_a = T$ as required.
\end{proof}

Finally we extend this result so that $X_a=T$ everywhere.

\begin{lem}\label{kleqn3}
$X_a = T$ for $a \leq n/3$ and $5 \leq n$ or $n$ even.
\end{lem}

\begin{proof}
Firstly, if $a = n/3$ then $\vec{\alpha}_{(a,a)} = (0,a, 0, 2a, 0, 0)$, in which case the following is a tying election table:

\begin{equation}\label{eq:3an}\begin{aligned}
\left(\begin{array}{c c c}
(a) \\
\hline
2 & -1 & -1 \\
-1 & 2 & -1 \\
-1 & -1 & 2
\end{array}\right)
\end{aligned}\end{equation}

Each row has $\vec{\alpha}(w, M_i) = \vec{\alpha}_{(a,a)}$, so $X_a=T$.

On the other hand, if $n/4 < a < n/3$ then the following is an $X_a$-transitive election table:

\begin{equation}\begin{aligned}
\left(\begin{array}{c c | c c | c}
(\frac{n-3a}{2}) &  & (\frac{5a-n}{2}) & & (n-2a) \\
\hline
+2 & +1 & +2 & -1 & -1 \\
+1 & +2 & -1 & +2 & -1
\end{array}\right)
\end{aligned}\end{equation}

The multiplicity of each column is a non-negative integer since $n/4 \leq a <n/3$, and each row is a copy of $\vec{\alpha}_{(a,a)}$, so this is indeed an $X_a$-transitive election table. This gives that $\vec{\alpha}(w,M_1+M_2) = \vec{\beta}_a = (n-3a, 0, 5a-n, 0, n-2a, 0)$ has $f'(\vec{\beta}_a)=T$.

Now we construct a further $X_a$-transitive election table:

\begin{equation}\begin{aligned}
\left(\begin{array}{c c | c c | c}
(n-3a) & & (a) & & (4a-n) \\
\hline
+3 & -2 & +1 & -2 & +1 \\
-2 & +3 & -2 & +1 & +1
\end{array}\right)
\end{aligned}\end{equation}

Again, the multiplicities of each column are non-negative integers, and each row is a copy of $\vec{\beta}_a$, so this is an $X_a$-transitive table. This gives that $\vec{\alpha}(w,M_1+M_2) = \vec{\gamma}_a = (0, 4a-n, 2n-6a, 2a, 0, 0) = \vec{\alpha}_{(4a-n,4a-n)}$ has $f'(\vec{\gamma}_a) = X_a$; in other words, $X_a = X_{4a-n}$. Since $a < n/3$ we know that $4a-n < a$. Let $a_1 = 4a-n$ and $a_{i+1}=4a_i-n$ for as long as $a_i > n/4$; this is a decreasing sequence of integers so it must eventually terminate with $a_{i+1} \leq n/4$. Now $X_{a_{i+1}} = T$ by Lemma \ref{kleqn4}. But we have established that $X_{a_{i+1}}=X_{a_i}$ for all $i$, so $X_a = T$ as required. This completes the proof.
\end{proof}

We clean up some loose ends and restate Lemma \ref{kleqn3} in a more natural way:

\begin{cor}\label{b3}
If $\vec{\alpha} \in B'$ has $\alpha_3 = \alpha_{-3} = 0$ then $f'(\vec{\alpha})=T$.
\end{cor}

\begin{proof}
For $n \geq 5$ or $n$ even it suffices to show that $\vec{\alpha} = \vec{\alpha}_{(p,q)}$ or $\varrho(\vec{\alpha}_{(p,q)})$ for some $p$ and $q$. For $\alpha_2 \geq \alpha_{-2}$, set $p = \alpha_2$ and $q = \alpha_2 - \alpha_{-2}$; then $\vec{\alpha}$ agrees with $\vec{\alpha}_{(p,q)}$ on $\alpha_2$ and $\alpha_{-2}$; $\alpha_1$ and $\alpha_{-1}$ are the determined uniquely from $\alpha_2$ and $\alpha_{-2}$ by $\sum_i \alpha_i = n$ and $\sum_i i\alpha_i = 0$, so the vectors are equal. Similarly for $\alpha_{-2} \geq \alpha_2$ set $p = \alpha_{-2}$ and $q = \alpha_{-2} - \alpha_2$; then $\vec{\alpha}$ agrees with $\varrho(\vec{\alpha}_{(p,q)}$ by the same argument.

It remains to show that $|\alpha_2 - \alpha_{-2}| \leq n/3$. But this is trivial as otherwise $\alpha_1+\alpha_{-1} < 2n/3$, so $|\alpha_1 - \alpha_{-1}| <2n/3$. Thus $2|\alpha_2 - \alpha_{-2}| < 2n/3$, contradicting our assumption that $|\alpha_2 - \alpha_{-2}| > n/3$.

Now if $n=3$ we must have $\alpha_2+\alpha_{-2}$ odd, but $\alpha_2+\alpha_{-2} \neq 3$ since $\sum_i i\alpha_i=0$. Thus either $\alpha_2 = 1$ and then $\alpha_{-1}=2$ or $\alpha_{-2}=1$ and $\alpha_1=2$; so we are in the case of $a=n/3$ or the reverse of such a vector; this is dealt with separately in Lemma \ref{kleqn3}, where we note that equation (\ref{eq:3an}) is a valid tying election table even for $a=1$.

Finally if $n=1$ then $B'$ is empty.
\end{proof}

This can now be extended to all of $B'$ in a few steps.

\begin{lem}\label{almostb4}
If $\vec{\alpha} \in B_n$ and $\alpha_3 + \alpha_{-3} \neq 1$ then $f'(\vec{\alpha})=T$.
\end{lem}

\begin{proof}
We will construct a $T$-transitive election table as follows, with $\vec{\gamma}$ to be chosen later (this table is transposed due to its length):

\begin{equation}\label{eq:almostb4tet}
\begin{aligned}
\left(\begin{array}{c | c c }
(\gamma_3) & 2 & 1 \\
(\alpha_3 - \gamma_3) & 1 & 2 \\
(\alpha_2 - 2\delta_2) & 1 & 1 \\
(\delta_2) & 3 & -1 \\
(\delta_2) & -1 & 3 \\
(\gamma_1) & 2 & -1 \\
(\alpha_1 - \gamma_1) & -1 & 2 \\
(\gamma_{-1}) & 1 & -2 \\
(\alpha_{-1}-\gamma_{-1}) & -2 & 1 \\
(\alpha_{-2}-\delta_{-2}) & -1 & -1 \\
(\delta_{-2}) & -3 & 1 \\
(\delta_{-2}) & 1 & -3 \\
(\gamma_{-3}) -1 & -2 \\
(\alpha_{-3}-\gamma_{-3}) & -2 & -1
\end{array}\right)
\end{aligned}
\end{equation}

The only conditions on $\delta_2$ and $\delta_{-2}$ are that $2\delta_i \leq \alpha_i$, and this is always possible at the very least by setting $\delta_i=0$. We require $\gamma_i \leq \alpha_i$ for all $i$. Now

\begin{equation}\begin{aligned}
d_1(\vec{\alpha}(w, M_1)) &= (\alpha_3 + \alpha_2 - \alpha_1 -2\alpha_{-1} - \alpha_{-2} - 2\alpha_{-3}) \\
&+ (\gamma_3 + 3\gamma_1 + 3\gamma_{-1} + \gamma_{-3}) \\
&= (-\alpha_3 - 3\alpha_1 - 3\alpha_{-1} - \alpha_{-3})/2 \\
&+ (\gamma_3 + 3\gamma_1 + 3\gamma_{-1} + \gamma_{-3})
\end{aligned}\end{equation}

The second equality is due to $\sum_i \alpha_i = 0$. We need to pick $\gamma_i$ such that this value is zero; in other words, so that $\alpha_3 + 3 \alpha_1 + 3 \alpha_{-1} + \alpha_{-3} = 2\gamma_3 + 6\gamma_1 + 6\gamma_{-1} + 2\gamma_{-3}$. We do so by setting $\gamma_1=\gamma_{-1}=0$ and $\gamma_3 + \gamma_{-3} \in \set{0,1,2}$ such that $2(\gamma_3 + \gamma_{-3}) \equiv \alpha_3 + \alpha_{-3} \pmod 3$. Note that this is always possible as $\gamma_3 +\gamma_{-3}$ can take any value at most $\alpha_3+\alpha_{-3}$, so it can take any value modulo 3 unless $\alpha_3 + \alpha_{-3} <2$; but we have disallowed $\alpha_3 + \alpha_{-3} =1$, and if $\alpha_3+\alpha_{-3}=0$ then $\gamma_3 =\gamma_{-3}=0$ works.

Now $\alpha_3 + 3\alpha_1 + 3\alpha_{-1} + \alpha_{-3}$ is an even positive integer, so $\alpha_3 + 3\alpha_1 + 3\alpha_{-1} + \alpha_{-3} - 2(\gamma_3+3\gamma_{1} + 3\gamma_{-1} + \gamma_{-3})$ is a non-negative integer and a multiple of three. Decrease it in increments of three by adding one to either $\gamma_1$ or $\gamma_{-1}$, or adding three to $\gamma_3+\gamma_{-3}$. This is possible until $\gamma_1 = \alpha_1$, $\gamma_{-1} = \alpha_{-1}$ and $2(\gamma_{3}+\gamma_{-3}) \geq \alpha_3+\alpha_{-3}-4$; at this point, the value is a multiple of six and at most four, so it is a non-positive integer. Thus at some point in the process its value is zero as required.

Now $\vec{\alpha}(w,M_1)$ and $\vec{\alpha}(w,M_2)$ both have $\alpha_3 = \alpha_{-3} = 0$, so by Lemma \ref{b3} this is a $T$-transitive election table. Then $\vec{\alpha}(w,M_1+M_2)= \vec{\alpha}$ has $f'(\vec{\alpha})=T$ as required.
\end{proof}

To deal with the case of $\alpha_3 = 1$ and $\alpha_{-3}=0$ (which is identical to the case of $\alpha_3 = 0$ and $\alpha_{-3}=1$) we require additional work, in which we compose transitive election tables by concatenating smaller ones.

\begin{lem}\label{addingthrees}
For $\vec{\alpha} \in B'$ for $n \geq 3$ with $\alpha_3 = \alpha_{-3} = 0$, there exists a transitive election table $(w,M)$ with $m=2$, $\vec{\alpha}(w,M_1+M_2) =\vec{\alpha}$, and $\vec{\alpha}(w,M_i) = \vec{\beta}^{(i)}$ satisfying $\beta_3^{(i)}+\beta_{-3}^{(i)} \geq 1$.
\end{lem}

\begin{proof}
We reuse the transitive election table (\ref{eq:almostb4tet}), which allows us to construct such $\vec{\beta}^{(i)}$ with $\beta_3^{(i)}=\beta_{-3}^{(i)}=0$. Now if $\alpha_i$ for $i = \pm 2$ is two or more, we can set $\delta_i = 1 \leq \alpha_i/2$, in which case we meet the criteria above.

Suppose instead that $\alpha_2, \alpha_{-2} \leq 1$. Then there are three cases to consider without loss of generality.

If $\alpha_2 = 1$ and $\alpha_{-2}=0$ then we must have $\alpha_{-1} \geq 2$. Then we consider $\vec{\alpha}' = \vec{\alpha} - (0,1,0,2,0,0)$, construct a transitive election table using (\ref{eq:almostb4tet}), and append entries $(1,-3,2)$, $(1,2,-3)$ and $(1,1,1)$ to $M_1$, $M_2$ and $w$ respectively, to give the required vectors. A symmetrical argument applies if $\alpha_2 = 0$ and $\alpha_{-2}=1$.

On the other hand, if $\alpha_2 = \alpha_{-2}=1$, then since $n\geq 3$ we must also have $\alpha_1 = \alpha_{-1} \geq 1$; then we consider $\vec{\alpha}' = \vec{\alpha} - (0,1,1,1,1,0)$, construct a transitive election table using (\ref{eq:almostb4tet}), and append entries $(3,2,-2, -3)$, $(-1,-3,3,1)$ and $(1,1,1,1)$ to $M_1$, $M_2$ and $w$ respectively to give the required vectors.

Finally, if $\alpha_2 = \alpha_{-2}=0$, then since $n \geq 3$ we must also have $\alpha_1 = \alpha_{-1} \geq 2$; then we consider $\vec{\alpha}' = \vec{\alpha} - (0,0,2,2,0,0)$, construct a transitive election table using (\ref{eq:almostb4tet}), and add entries $(3,-3,2, -2)$, $(-2,2,-3,3)$ and $(1,1,1,1)$ to $M_1$, $M_2$ and $w$ respectively to give the required vectors.

This completes the proof.
\end{proof}

We can now use the constructions given by Lemmas \ref{almostb4} and \ref{addingthrees} to decompose vectors $\vec{\alpha} \in B'$ that were not covered by Corollary \ref{almostb4} into ones that were. There are a few separate cases of this.

\begin{lem}\label{alphaminus1alphaminus2}
If for $n\geq 6$, $\vec{\alpha} \in B'$ has:
\begin{itemize}
\item $\alpha_3 = 1$;
\item $\alpha_{-3}=0$; and
\item $\alpha_{-1}, \alpha_{-2}\geq 1$
\end{itemize}
then $f'(\vec{\alpha})=T$.
\end{lem}

\begin{proof}
We consider $\vec{\alpha}' = \vec{\alpha} - (1,0,0,1,1,0)$. $\vec{\alpha}' \in B'_{n-3}$ and $n-3 \geq 3$, so we apply Lemma \ref{addingthrees} and get a transitive election table with rows $M_1$ and $M_2$. We add entries $(-3,2,1)$, $(2,1,-3)$ and $(1,1,1)$ to $M_1$, $M_2$ and $w$ respectively to give a new transitive election table with rows $M'_1$ and $M'_2$, with $\vec{\alpha}(w,M'_1+M'_2) = \vec{\alpha}' + (1,0,0,1,1,0) =\vec{\alpha}$. But each $\vec{\alpha}(w,M'_i)$ has at least two $\pm 3$ entries; one from the construction in Lemma \ref{addingthrees} and one that we just added to go from $M_i$ to $M'_i$. So by Corollary \ref{almostb4} we know that both $\vec{\alpha}(w,M'_i)$ are tying vectors under $f'$, so they are rows of a $T$-transitive election table, which completes the proof.
\end{proof}

\begin{lem}\label{alphaminus1}
If $\vec{\alpha} \in B'$ has:
\begin{itemize}
\item $\alpha_3 = 1$;
\item $\alpha_{-3}=0$; and
\item $\alpha_{-2}=0$
\end{itemize}
then $f'(\vec{\alpha})=T$.
\end{lem}

\begin{proof}
Note that $\vec{\alpha} = (1, \alpha_2, \alpha_1, 3+2\alpha_2+\alpha_1,0,0)$. If $\alpha_2 = 0$ then the following is an $f'(\vec{\alpha})$-transitive election table:

\begin{equation}\begin{aligned}
M &= \left(\begin{array}{c c c c | c c}
(1) & & & & (\alpha_1) \\
\hline
+3 & -1 & -1 & -1 & +1 & -1 \\
-1 & +3 & -1 & -1 & +1 & -1 \\
-1 & -1 & +3 & -1 & -1 & +1 \\
-1 & -1 & -1 & +3 & -1 & +1
\end{array}\right) \\ \\
P &= ((x_1x_2)(x_3x_4))
\end{aligned}\end{equation}

Thus $X=T$ as required. Otherwise $\alpha_2 \geq 1$ and we note that $\alpha_{-1} \geq 5$. Then we consider $\vec{\alpha}' = \vec{\alpha} - (1,1,0,5,0,0)$. We can construct a transitive election table $(w,M)$ as in Lemma \ref{almostb4}. We add entries $(2,-1,-3,-3,2,2,1)$, $(1,3,2,2,-3,-3,-2)$ and $(1,1,1,1,1,1,1)$ to $M_1$, $M_2$ and $w$ respectively to give $M'_1$ and $M'_2$ with $\vec{\alpha}(w, M'_1+M'_2) = \vec{\alpha}' + (1,1,0,5,0,0) =\vec{\alpha}$. But each $M'_i$ has at least two $\pm 3$ entries, both added among the additional entries used to go from $M_i$ to $M'_i$. So by Corollary \ref{almostb4} we know that both $\vec{\alpha}(w,M_i)$ are tying vectors under $f'$, so they are rows of a $T$-transitive election table, which completes the proof.
\end{proof}

\begin{lem}\label{alphaminus2}
If $\vec{\alpha} \in B'$ has:
\begin{itemize}
\item $\alpha_3 = 1$;
\item $\alpha_{-3}=0$; and
\item $\alpha_{-1}=0$
\end{itemize}
then $f'(\vec{\alpha})=T$.
\end{lem}

\begin{proof}
Note that $2\alpha_{-2} = 3 + 2\alpha_2+\alpha_1 \geq 3$, so $\alpha_{-2}\geq 2$. On the other hand, both sides of this equation are even so we find that $\alpha_1 \geq 1$. Thus we can consider $\vec{\alpha}' = \vec{\alpha}-(1,0,1,0,2,0)$. We construct a transitive election table $(w,M)$ as in Lemma \ref{almostb4}. We add entries $(1,3,-3,-1)$, $(2,-2,1,-1)$ and $(1,1,1,1)$ to $M_1$, $M_2$ and $w$ respectively to give $M'_1$ and $M'_2$ such that $\vec{\alpha}(w,M'_1+M'_2) = \vec{\alpha}' + (1,0,1,0,2,0) =\vec{\alpha}$. Now $M'_1$ has two $\pm 3$-entries and $M'_2$ has none, so by Lemma \ref{almostb4} we know that both $\vec{\alpha}(w,M_i)$ are tying vectors under $f'$, so they are rows of a $T$-transitive election table, which completes the proof.
\end{proof}

Only one small case remains:

\begin{lem}
If $n \leq 5$ and $\vec{\alpha} \in B_n$ has $\alpha_3 = 1$, $\alpha_{-3}=0$ and $\alpha_{-1}, \alpha_{-2} \geq 1$, then $f'(\vec{\alpha})=T$.
\end{lem}

\begin{proof}
Consider $\vec{\alpha}' = \vec{\alpha} - (1, 0, 0, 1, 1, 0)$; this has $\sum_i \alpha'_i \leq 2$ and $\sum_i i \alpha'_i = 0$. Moreover $\alpha'_3 = \alpha'_{-3}=0$ so we must have one of the following:

\begin{enumerate}
\item $\vec{\alpha}' = (0,0,1,1,0,0)$;
\item $\vec{\alpha}' = (0,1,0,0,1,0)$; or
\item $\vec{\alpha}' = 0$
\end{enumerate}

In the first case, the following is an $X$-transitive election table where each row has $\vec{\alpha}(w, M_i)=\vec{\alpha}$:

\begin{equation}\begin{aligned}
\left(\begin{array}{c c c c c}
+3 & +1 & -1 & -1 & -2  \\
-2 & +1 & -1 & -1 & +3 
\end{array}\right)
\end{aligned}\end{equation}

But $\vec{\alpha}(w, M_1+M_2) = (0,1,2,0,2,0)$ which has $f'(\vec{\alpha}(w, M_1+M_2)) = T$ by Corollary \ref{b3}, so $f'(\vec{\alpha})=T$ as well.

But the following is an $T$-transitive election table where each row has $\vec{\alpha}(w, M_i)=\varrho(\vec{\alpha})$:

\begin{equation}\begin{aligned}
\left(\begin{array}{c c c c c}
+2 & +1 & -3 & +1 & -1  \\
+1 & +1 & +2 & -3 & -1 
\end{array}\right)
\end{aligned}\end{equation}

Thus $\vec{\alpha}(w, M_1+M_2) = (1,1,0,1,2,0)$ has $f'(\vec{\alpha}(w, M_1+M_2))=T$, which resolves the second case.

Finally, in the third case, the following is a tying election table:

\begin{equation}\begin{aligned}
\left(\begin{array}{c c c}
+3 & -1 & -2 \\
-1 & -2 & +3 \\
-2 & +3 & -1
\end{array}\right)
\end{aligned}\end{equation}

So $f'(\vec{\alpha}) = f'(\vec{\alpha}(w,M_i))=T$ by Corollary \ref{tyingelectiontable} as required.
\end{proof}

We have covered all cases of $\vec{\alpha} \in B'$:

\begin{cor}\label{b4}
If $\vec{\alpha} \in B'$ then $f'(\vec{\alpha})=T$.
\end{cor}

Now we address all vectors in $A'$. For $n$ odd we define $(0,0,\frac{n-1}{2},\frac{n+1}{2},0,0) = \vec{\beta}$; for $n$ even we define $(0,0,\frac{n}{2},\frac{n}{2}-1,1,0) = \vec{\beta}$. We then define $f'(\vec{\beta})=X_-$.

\begin{lem}\label{firstm1iter}
For $n$ odd, if $\vec{\alpha} \in A'$ with
\begin{itemize}
\item $d_1(\vec{\alpha}) = -1$;
\item $\alpha_3+\alpha_{-1} \leq \frac{n-1}{2}$; and
\item $\alpha_1 + \alpha_{-3} \leq \frac{n+1}{2}$
\end{itemize}
then $f'(\vec{\alpha}) = X_-$.

Moreover, for $n$ even, if $\vec{\alpha} \in A'$ with
\begin{itemize}
\item $d_1(\vec{\alpha}) = -1$;
\item $\alpha_3+\alpha_{-1} \leq \frac{n}{2}$;
\item $\alpha_1 + \alpha_{-3} \leq \frac{n}{2}$; and
\item $\alpha_1 + \alpha_{-1} + \alpha_{-3} \geq 1$
\end{itemize}
then $f'(\vec{\alpha}) = X_-$.
\end{lem}

\begin{proof}
In the odd case, note that $0 \leq \frac{n-1}{2} - \alpha_3 - \alpha_{-1}$ and that $0 \leq \frac{n+1}{2} - \alpha_1 - \alpha_{-3}$. Moreover note that these two values sum to $\alpha_2 + \alpha_{-2}$; so let $\gamma_2$ and $\gamma_{-2}$ be such that $\gamma_2+\gamma_{-2} = \frac{n-1}{2} - \alpha_3 - \alpha_{-1}$ and $0 \leq \gamma_i \leq \alpha_i$. Then the following is a transitive election table:

\begin{equation}\begin{aligned}
\left(\begin{array}{c c c c c c c c}
(\alpha_3) & (\alpha_1) & (\alpha_{-1}) & (\alpha_{-3}) & (\gamma_2) & (\gamma_{-2}) & (\alpha_2 - \gamma_2) & (\alpha_{-2} - \gamma_{-2}) \\
\hline
+1 & -1 & +1 & -1 & +1 & +1 & -1 & -1 \\
+2 & +2 & -2 & -2 & +1 & -3 & +3 & -1
\end{array}\right)
\end{aligned}\end{equation}

Here $\vec{\alpha}(w,M_1) = \vec{\beta}$ and $d_1 \circ \vec{\alpha}(w,M_2) = 0$, so $f'(\vec{\alpha}(w,M_1)) = X_-$ and $f'(\vec{\alpha}(w,M_2)) = T$, and we conclude that $f'(\vec{\alpha}) = f'(\vec{\alpha}(w,M_1+M_2) = X_-$ as required.

In the even case, if $\alpha_1 + \alpha_{-3} = 0$ then $\alpha_{-1} = \alpha_1 + \alpha_{-1} + \alpha_{-3} \geq 1$, and set $j=-1$. Otherwise select $j \in \set{1,-3}$ such that $\alpha_j \geq 1$. Then let $\alpha'_j = \alpha_j-1$ and $\alpha'_i = \alpha_i$ elsewhere. Additionally, note that $0 \leq \frac{n}{2} - \alpha_3 - \alpha_{-1}$ and $0 \leq \frac{n}{2} - \alpha_1 - \alpha_{-3}$, and that these two values sum to $\alpha_2 + \alpha_{-2}$; so let $\gamma_2$ and $\gamma_{-2}$ be such that $\gamma_2+\gamma_{-2} = \frac{n}{2} - \alpha_3 - \alpha_{-1}$ and $0 \leq \gamma_i \leq \alpha_i$.

Now if $j=-1$ then we know that $\alpha_1 + \alpha_{-3}=0$, so $\alpha_2 + \alpha_{-2} + \alpha_3 + \alpha_{-1} =n$, and $\alpha_2+\alpha_{-2} = \gamma_2 + \gamma_{-2} + n/2 \geq \gamma_2 + \gamma_{-2} + 1$. Therefore one of $\gamma_2$ and $\gamma_{-2}$ can be increased to $\gamma'_i = \gamma_i+1$ while maintaining $\gamma'_i \leq \alpha_i$, and we let $\gamma'_i = \gamma_i$ elsewhere. Now $\gamma'_2 + \gamma'_{-2} = \frac{n}{2} - \alpha_3 - \alpha'_{-1}$; the same is true if $j\neq -1$ as in this case $\alpha_1 = \alpha'_{-1}$. Then consider the following transitive election table:

\begin{equation}\begin{aligned}
\left(\begin{array}{c c c c c c c c c}
(\alpha_3) & (\alpha'_1) & (\alpha'_{-1}) & (\alpha'_{-3}) & (\gamma'_2) & (\gamma'_{-2}) & (\alpha_2 - \gamma'_2) & (\alpha_{-2} - \gamma'_{-2}) & (1) \\
\hline
+1 & -1 & +1 & -1 & +1 & +1 & -1 & -1 & -2 \\
+2 & +2 & -2 & -2 & +1 & -3 & +3 & -1 & j+2
\end{array}\right)
\end{aligned}\end{equation}

Here $\vec{\alpha}(w,M_1) = \vec{\beta}$ and $d_1 \circ \vec{\alpha}(w,M_2) = 0$, so $f'(\vec{\alpha}(w,M_1)) = X_-$ and $f'(\vec{\alpha}(w,M_2))=T$, and we conclude that $f'(\vec{\alpha}) = f'(\vec{\alpha}(w,M_1+M_2) = X_-$ as required.
\end{proof}

The following Lemma removes the final condition in the even case:

\begin{lem}\label{evenfix}
For $n$ even, if $\vec{\alpha} \in A'$ with
\begin{itemize}
\item $d_1(\vec{\alpha})=-1$; and
\item $\alpha_1 + \alpha_{-1}+\alpha_{-3} = 0$
\end{itemize}
then $f'(\vec{\alpha}) = X_-$.
\end{lem}

\begin{proof}
Note that $\alpha_3$ is odd. Let $\gamma = \lceil \frac{n/2 - \alpha_2 - \alpha_3}{2} \rceil$. Then $n - 2\gamma - \alpha_2 - \alpha_3 \geq \frac{n}{2} - 1 \geq 0$, so the following is a transitive election table:

\begin{equation}\begin{aligned}
\left(\begin{array}{c c c c c c}
(\frac{\alpha_3+1}{2}) & (\frac{\alpha_3-1}{2}) & (\alpha_2) & (n - 2\gamma - \alpha_2 - \alpha_3) & (\gamma) & (\gamma) \\
\hline
+2 & +1 & +1 & -1 & +1 & -3 \\
+1 & +2 & +1 & -1 & -3 & +1
\end{array}\right)
\end{aligned}\end{equation}

because $\sum_i w_i = n$ and all $w_i \geq 0$, and $d_1(\vec{\alpha}(w,M_1))=0$ so $f'(\vec{\alpha}(w,M_1))=T$ by Corollary \ref{b4}. On the other hand, letting $\vec{\alpha}(w,M_2)) = \vec{\beta}$ we have $d_1(\vec{\beta}) = -1$, so it suffices to check that $\vec{\beta}$ satisfies the conditions for Lemma \ref{firstm1iter} for $n$ even. We have $\beta_1+\beta_{-1} + \beta_{-3} \geq \beta_1 \geq 1/2>0$ as required, and we require
\begin{itemize}
\item $\alpha_2 + \frac{\alpha_3+1}{2} + 2\gamma = \beta_1 + \beta_{-3} \leq \frac{n}{2}$ and
\item $n-2\gamma - \alpha_2 - \alpha_3 \leq \frac{n}{2}$
\end{itemize}
The second inequality is immediately true by our choice of $\gamma$. But our choice of $\gamma$ also guaranteed that $2\gamma \leq n/2 - \alpha_2 - \alpha_3 + 2$. Then the first inequality holds as long as $\frac{1-\alpha_3}{2} + 2 \leq 0$, which is true for $\alpha_3 \geq 3$. Otherwise we have $\alpha_3=1$. This means that $-1 = d_1(\vec{\alpha}) = 3+2\alpha_2-2\alpha_{-2}$, so $\alpha_2 - \alpha_{-2} = -2$, and $n = 1 + 2\alpha_2 + 2$ which is odd, a contradiction. Thus both inequalities hold, and we can apply Lemma \ref{firstm1iter} to find that $f'(\vec{\beta}) = X_-$, so the above is an $X_-$-transitive election table and $f'(\vec{\alpha}) = f'(\vec{\alpha}(w,M_1+M_2)) = X_-$ as required.
\end{proof}

\begin{lem}\label{secondm1iter}
If $\vec{\alpha} \in A'$ with
\begin{itemize}
\item $d_1(\vec{\alpha}) = -1$;
\item $\alpha_3 + \alpha_{-1} \geq 1$;
\item $\alpha_1 + \alpha_{-3} \geq 1$; and
\item $\alpha_2 + \alpha_{-2} \geq 1$
\end{itemize}
then $f'(\vec{\alpha}) = X_-$.
\end{lem}

\begin{proof}
For $n$ odd, Lemma \ref{firstm1iter} covers all but the cases where either $\alpha_3+\alpha_{-1}>\frac{n+1}{2}$ (which we will call the positive case) or $\alpha_1+\alpha_{-3} > \frac{n+3}{2}$ (the negative case). For $n$ even, we apply either Lemma \ref{firstm1iter} or Lemma \ref{evenfix} unless we are in either the positive case, which becomes $\alpha_3 + \alpha_{-1} >\frac{n}{2}$, or the negative case, which becomes $\alpha_1 + \alpha_{-3} > \frac{n}{2}$. The arguments for the positive and negative cases are almost identical.

Let $\vec{e} \in A$ be such that $\tau_A(\vec{e})=\vec{\alpha}$, and for $i \in D_4$ let $A_i = \set{j: e_j=i}$. Now we define the functions $u_{\max}^+$, $u_{\max}^-$, $u_{\min}^+$ and $u_{\min}^-$ as follows:

\begin{equation}\begin{aligned}
\begin{matrix}
i & u_{\max}^+(i) & u_{\max}^-(i) & u_{\min}^+(i) & u_{\min}^-(i) \\
\hline
+3 & +2 & +2 & +2 & +1 \\
+2 & +3 & +3 & -1 & -1 \\
+1 & +3 & +2 & -2 & -2 \\
-1 & +2 & +2 & -2 & -3 \\
-2 & +1 & +1 & -3 & -3 \\
-3 & -1 & -2 & -2 & -2
\end{matrix}
\end{aligned}\end{equation}

Now we define $\vec{u}_{\max}$ and $\vec{u}_{\min}$ such that for $j \in A_i$ we have $u_{\max,j} = u_{\max}^+(i)$ in the positive case and $u_{\max}^-(i)$ in the negative case, and similarly for $\vec{u}_{\min}$. Note that $\vec{e}, \vec{u}_{\max}, \vec{u}_{\min}, \vec{e}+\vec{u}_{\max}, \vec{e}+\vec{u}_{\min} \in D_4^n$. We will define an intermediate $\vec{u} \in D_4^n$ for which $d_1 \circ \tau_A(\vec{u})=\sum_i u_i = -1$, by altering $\vec{u}_{\max}$ incrementally until it reached $\vec{u}_{\min}$, and halting somewhere in that process.

First note that in the positive case, $\sum_j u_{\max,j} \geq -1$ unless $\alpha_{-3} > \alpha_{-2}+\alpha_{-1}+\alpha_1+\alpha_2+\alpha_3$, which is not possible as in that case $d_1(\vec{\alpha})\leq -3 < -1$. Similarly in the negative case $\sum_j u_{\min,j} \leq -1$ unless $\alpha_3 \geq \alpha_2+\alpha_1+\alpha_{-1}+\alpha_{-2}+\alpha_{-3}$, which is not possible as in that case $d_1(\vec{\alpha})\geq 0>-1$.

Moreover, in the positive case, $\sum_j u_{\min,j} \leq -1$ unless $\alpha_{-3} > \alpha_3 + 3\alpha_2+3\alpha_1+\alpha_{-1}+\alpha_{-2}$, which is not possible as in that case $d_1(\vec{\alpha})\leq -3 < -1$. Similarly, in the negative case, $\sum_j u_{\max,j} \geq -1$ unless $\alpha_3 > \alpha_{-3}+3\alpha_{-2}+3\alpha_1+\alpha_{-1}+\alpha_{-2}$, which is not possible as in that case $d_1(\vec{\alpha})\geq 3>-1$.

Therefore in both cases, by varying $\vec{u}$ to be between $\vec{u}_{\max}$ and $\vec{u}_{\min}$ at each entry, we can vary $d_1 \circ \tau_A(\vec{u})$ over a range that includes $-1$. In order to check that we can indeed achieve $-1$, we need to consider the increments that can be added to or subtracted from $d_1 \circ \tau_A(\vec{u})$. We will proceed in the positive case; for the negative case, we do the negative version of everything (so we start with $\vec{u}_{\min}$ and incrementally increase $\vec{u}$). Suppose that we start at $\vec{u}_{\max}$ and the total decrease required is $\sum_i u_{\max,i} + 1 = d$.

We let $X = \set{1,-3}$ in the positive case (so $X=\set{3,-1}$ in the negative case). If $d$ is odd, then since $\sum_{x \in X} \alpha_x \geq 1$ we can pick $i_1 \in \cup_{x \in X} A_x$ and decrease $u_{i_1}$ by $1$. This updates the value of $d$ to $d-1$ even.

Now we pair up all $i \in \cup_{x \in X} A_x$ (excluding $i_1$ if it exists) into pairs $(j_1, j_2) \in P_1$ and promise that we will only reduce $u_{j_1}$ by $1$ if we are also reducing $u_{j_2}$ by $1$. With the possible exception of a single $i^{(1)}$ which cannot be paired, this means that all of these indices can now be treated as new indices $j \in P_1$ which permit a reduction of $\vec{u}$ by $2$.

Now if $d$ is even but not a multiple of $4$, then since $\alpha_2+\alpha_{-2} \geq 1$ we can pick $i_2 \in A_2 \cup A_{-2}$ and decrease $u_{i_2}$ by $2$ (from $3$ to $1$ if $i_2 \in A_2$ or from $-1$ to $-3$ if $i_2 \in A_{-2}$). This updates the value of $d$ to $d-2$ a multiple of $4$.

Now we pair up all $i \in P_1$, along with $i_2$ if it exists, into pairs $(j_1,j_2) \in P_2$ and promise that we will only reduce $u_{j_1}$ by $2$ if we are also reducing $u_{j_2}$ by $2$. With the possible exception of a single $i^{(2)}$ which cannot be paired, this means that all of these indices can now be treated as new indices $j \in P_2$ which permit a reduction of $\vec{u}$ by $4$.

Now $d$ can be reduced by $4$ by any of the following changes to $\vec{u}$:

\begin{itemize}
\item reducing $u_i$ from $2$ to $-2$ for $i \in A_{-1}$ in the positive case or $i \in A_1$ in the negative case;
\item reducing $u_i$ from $3$ to $-1$ for $i \in A_2$;
\item reducing $u_i$ from $3$ to $-1$ or from $2$ to $-2$ for $i \in A_1$ in the positive case, or from $2$ to $-2$ or from $1$ to $-3$ for $i \in A_{-1}$ in the negative case;
\item reducing $u_i$ from $1$ to $-3$ for $i \in A_{-2}$;
\item reducing $u_j$ by $4$ for $j \in P_2$.
\end{itemize}

After completing all of these reductions, and further reducing $u_{i^{(1)}}$ by $1$ if it exists and $u_{i^{(2)}}$ by $2$ if it exists, we must be in a case where $u_j = u_{\min}(i)$ for $j \in A_i$. At this point, $\sum_i u_i \leq -1$. Now if we had not carried out the reductions of $u_{i^{(1)}}$ and  $u_{i^{(2)}}$, we would have had a value at most three more than this; so in this case we would have $\sum_i u_i \leq 2$. On the other hand, we would never have deviated from $d$ being a multiple of $4$, so we would have $\sum_i u_i \equiv -1 \pmod 4$. So indeed $\sum_i u_i \leq -1$ and we must have had some stage in the reductions of $d$ where $d=0$; for this $\vec{u}$ we have $d_1(\vec{u})=-1$, $d_1(\vec{e})=-1$, and $\vec{u}-\vec{e} \in D_4^n$ with $d_1(\vec{e}-\vec{u})=0$. Thus by Corollary \ref{b4} we know that $f' \circ \tau_A(\vec{e}-\vec{u})=T$.

On the other hand, letting $\vec{\beta} = \tau_A(\vec{u})$ we have $\beta_3 + \beta_1 + \beta_{-1} + \beta_{-3} \leq n - \sum_{x \in \set{-3,-1,1,3} \setminus X} \alpha_x \leq \frac{n-1}{2}$, so $\beta_3 + \beta_{-1} \leq \frac{n-1}{2}$ and $\beta_1 + \beta_{-3} \leq \frac{n+1}{2}$. Therefore for $n$ odd Lemma \ref{firstm1iter} applies and $f'(\vec{\beta}) = X_-$; and therefore $f'(\vec{\alpha})=X_-$ as required. For $n$ even, Lemma \ref{firstm1iter} still applies unless $\beta_1 + \beta_{-1}+\beta_{-3}=0$, in which case we apply Lemma \ref{evenfix}.
\end{proof}

\begin{lem}\label{thirdm1iterA}
If $\vec{\alpha} \in A'$ with
\begin{itemize}
\item $d_1(\vec{\alpha}) = -1$;
\item $\alpha_2 + \alpha_{-2} \geq 1$; and
\item $\alpha_i + \alpha_{-3i} = 0$ for some $i=\pm 1$
\end{itemize}
then $f'(\vec{\alpha}) = X_-$.
\end{lem}

\begin{proof}
In the argument above, in the case where $\alpha_{-i} + \alpha_{3i}$ is too large for Lemma \ref{firstm1iter} to apply, we only required that $\alpha_i + \alpha_{-3i}>0$ in order to adjust $d$ to be even. Suppose instead that $\alpha_i + \alpha_{-3i} = 0$. For any $j \in A_{-i} \cup A_{3i}$ we can reduce $u_j$ (or increase it in the negative case) by $1$ so that $d$ is even; then the proof proceeds normally. The only case in which this doesn't work is if the resulting $\vec{\beta}$ has $\beta_{-i}+\beta_{3i}$ too large to apply Lemma \ref{firstm1iter}. The cases where this could occurs are as follows:

\begin{itemize}
\item $\alpha_1 + \alpha_{-3}=0$, $n$ odd, $\beta_1 + \beta_{-3} \geq \frac{n+3}{2}$
\item $\alpha_3 + \alpha_{-1}=0$, $n$ odd, $\beta_3 + \beta_{-1} \geq \frac{n+1}{2}$
\item $\alpha_1 + \alpha_{-3}=0$, $n$ even, $\beta_1 + \beta_{-3} \geq \frac{n}{2}+1$
\item $\alpha_3 + \alpha_{-1}=0$, $n$ even, $\beta_3 + \beta_{-1} \geq \frac{n}{2}+1$
\end{itemize}

In the first case, $\beta_1 + \beta_{-3} \leq \frac{n-1}{2}+1$, a contradiction.
In the second case, $\beta_3 + \beta_{-1} \leq \frac{n-3}{2}+1$, a contradiction.
In the third and fourth case, $\beta_1 + \beta_{-3} \leq \frac{n}{2}$, a contradiction.
So this alternative $\vec{u}$ always works, and the lemma holds.
\end{proof}

If $\alpha_2+\alpha_{-2} = n$ then $d_1(\vec{\alpha})$ is even, a contradiction; so we have covered all cases in which $\alpha_2+\alpha_{-2} >0$. Next we deal with cases where $\alpha_2+\alpha_{-2} =0$.

\begin{lem}\label{thirdm1iter}
If $\vec{\alpha} \in A'$ with
\begin{itemize}
\item $d_1(\vec{\alpha}) = -1$;
\item $\alpha_1 + \alpha_{-3} > 0$; and
\item $\alpha_2 + \alpha_{-2} = 0$
\end{itemize}
then $f'(\vec{\alpha}) = X_-$.
\end{lem}

\begin{proof}
First note that in this case we must have $n$ odd. We let $\vec{\beta} = (0,\frac{n-1}{2},0,1,\frac{n-1}{2},0)$. Moreover let $\vec{b} \in D_4^n$ be defined as follows: $b_i = 2$ for $i \leq \frac{n-1}{2}$, $b_i = -2$ for $\frac{n-1}{2} < i \leq n-1$, and $b_n = -1$. Clearly $\tau_A(\vec{b})=\vec{\beta}$.

Furthermore, define $\vec{a}$ by

\begin{equation}\begin{aligned}
a_i = \begin{cases}
+3 & i \leq \alpha_3 \\
-1 & \alpha_3 < i \leq \alpha_3 + \alpha_{-1} \\
+1 & \alpha_3 + \alpha_{-1} < i \leq \alpha_3 + \alpha_{-1} + \alpha_1 \\
-3 & \alpha_3 + \alpha_{-1} + \alpha_1 < i
\end{cases}
\end{aligned}\end{equation}

Now $f' \circ \tau_A(\vec{b}) = X_-$ by Lemma \ref{firstm1iter}; moreover, $d_1 \circ \tau_A(\vec{a}) = d_1 \circ \tau_A(\vec{b}) = -1$, so if $\vec{a}-\vec{b} \in D_4^n$ then $d_1 \circ \tau_A(\vec{a}-\vec{b}) = 0$ and $f' \circ \tau_A(\vec{a}-\vec{b}) = T$ by Corollary \ref{b4}; this would imply that $f' \circ \tau_A(\vec{a}) = f'(\vec{\alpha})=X_-$. So all that's left is to check that indeed $\vec{a}-\vec{b} \in D_4^n$.

For $i \leq \frac{n-1}{2}$, $a_i - b_i \in D_4$ unless $a_i = -3$. But this would imply that $\alpha_{-3} > \frac{n+1}{2}$, so $d_1(\vec{\alpha}) < -3 <-1$, a contradiction.

For $\frac{n-1}{2} < i \leq n-1$, $a_i-b_- \in D_4$ unless $a_i = 3$. But this would imply that $\alpha_3 > \frac{n-1}{2}$, so $\alpha_3 \geq \frac{n-1}{2}$ and $d_1(\vec{\alpha}) \geq 3>-1$, a contradiction.

Finally, $a_n - b_n \in D_4$ unless $\alpha_3 + \alpha_{-1} = n$. In this case $\alpha_1 + \alpha_{-3}=0$, contradicting the assumption of the Lemma. So indeed $\vec{a}-\vec{b} \in D_4^n$ and $f'(\vec{\alpha})=X_-$ as required.
\end{proof}

This leaves a final case.

\begin{lem}
If $\vec{\alpha} \in A'$ with
\begin{itemize}
\item $d_1(\vec{\alpha}) = -1$; and
\item $\alpha_3+\alpha_{-1} = n$
\end{itemize}
then $f'(\vec{\alpha}) = X_-$.
\end{lem}

\begin{proof}
We repeat the argument from above, but let $\vec{\beta} = (0,\frac{n-1}{2}, 1, 0, \frac{n-5}{2}, 2)$. We have $f'(\vec{\beta})=X_-$ by Lemma \ref{thirdm1iterA}. Now $\vec{b}$ is defined such that $b_i = 2$ for $i \leq \frac{n-1}{2}$, $b_i = -2$ for $\frac{n-1}{2} < i \leq n-3$, $b_{n-2} = b_{n-1}=-2$ and $b_n = 1$. Moreover $\vec{a}$ has $a_i = 3$ for $i \leq \frac{n-1}{4}$ and $a_i = -1$ otherwise.

Now clearly $a_i-b_i \in D_4$ for $i \leq n-3$; for $i > n-3$ we only require that $a_i = -1$. Suppose to the contrary; then $\frac{n-1}{4}\geq n-2$, so $n \leq 2$. But $n \equiv 1 \pmod 4$ so $n=1$ and we are in the case of one voter, where $X_-$ was defined as $f'((0,0,0,1,0,0))$, which is exactly this case.

Thus we can conclude as above, and find that $f'(\vec{\alpha})=X_-$.
\end{proof}

We have demonstrated that $f'(\vec{\alpha})$ is constant over all $\vec{\alpha}$ with $d_1(\vec{\alpha})=-1$:

\begin{cor}\label{m1}
If $\vec{\alpha} \in A'$ with $d_1(\vec{\alpha})=-1$, then $f'(\vec{\alpha}) = X_-$.
\end{cor}

We now extend this to all $\vec{\alpha}$ with $d_1(\vec{\alpha})<0$.

\begin{lem}\label{mmain}
If $\vec{\alpha} \in A'$ with
\begin{itemize}
\item $d_1(\vec{\alpha}) \leq -2$;
\item $n \neq 2$
\end{itemize}
then $f'(\vec{\alpha}) = X_-$.
\end{lem}

\begin{proof}
We will prove this inductively, showing that any $\vec{\alpha}$ with $d_1(\vec{\alpha}) = m \leq -2$ can be composed of $\vec{\beta}_1$ and $\vec{\beta}_2$, where $d_1(\vec{\beta}_i) = m_i$ both have $m<m_i <0$. Corollary \ref{m1} gives that $f'(\vec{\beta}_i)=X_-$ if $m_i=-1$ as a base case.

Let $\vec{e}$ be such that $\tau_A(\vec{e})=\vec{\alpha}$, and for $i \in D_4$ let $A_i = \set{j: e_j=i}$. Now we define the functions $u_{\max}$ and $u_{\min}$ as follows:

\begin{equation}\begin{aligned}
\begin{matrix}
i & u_{\max}(i) & u_{\min}(i) \\
\hline
+3 & +2 & +1 \\
+2 & +3 & -1 \\
+1 & +3 & -2 \\
-1 & +2 & -3 \\
-2 & +1 & -3 \\
-3 & -1 & -2
\end{matrix}
\end{aligned}\end{equation}

Now we define $\vec{u}_{\max}$ and $\vec{u}_{\min}$ such that for $j \in A_i$ we have $u_{\max,j} = u_{\max}(i)$, and similarly for $\vec{u}_{\min}$.

Note that $\vec{e}, \vec{u}_{\max}, \vec{u}_{\min}, \vec{e}+\vec{u}_{\max}, \vec{e}+\vec{u}_{\min} \in D_4^n$. We will define an intermediate $\vec{u} \in D_4^n$ for which $d_1 \circ \tau_A(\vec{u})=\sum_i u_i = \lfloor m/2 \rfloor$, by altering $\vec{u}_{\max}$ incrementally until it reached $\vec{u}_{\min}$, and halting somewhere in that process. Then $\tau_A(\vec{u})$ and $\tau_A(\vec{e}-\vec{u})$ will give the required $\vec{\beta}_i$.

First note that $u_{\max}(i)\geq i/2 \geq u_{\min}(i)$ for all $i$, so $\sum_i u_{\max, i} \geq m/2 \geq \sum_i u_{\min, i}$. Therefore by varying $\vec{u}$ to be between $\vec{u}_{\max}$ and $\vec{u}_{\min}$ at each entry, we can vary $d_1 \circ \tau_A(\vec{u})$ over a range that includes $\lfloor m/2 \rfloor = k$. In order to check that we can indeed achieve $k$, we need to consider the increments that can be added to or subtracted from $d_1 \circ \tau_A(\vec{u})$. Suppose that we start at $\vec{u}_{\max}$ and the total decrease required is $\sum_i u_{\max,i} - k = d$.

If $\alpha_2 + \alpha_{-2} = n$ then we can in fact skip all of this work and simply set $\vec{u} = \vec{e}/2$; otherwise, if $d$ is odd, then since $\alpha_3+\alpha_1+\alpha_{-1}+\alpha_{-3} \geq 1$ we can pick $i_1 \in \cup_{2 \nmid x} A_x$ and decrease $u_{i_1}$ by $1$. This updates the value of $d$ to $d-1$ even.

Now we pair up all $i \in \cup_{2 \nmid x} A_x$ (excluding $i_1$ if it exists) into pairs $(j_1, j_2) \in P_1$ and promise that we will only reduce $u_{j_1}$ by $1$ if we are also reducing $u_{j_2}$ by $1$. With the possible exception of a single $i^{(1)}$ which cannot be paired, this means that all of these indices can now be treated as new indices $j \in P_1$ which permit a reduction of $\vec{u}$ by $2$.

Now if $d$ is even but not a multiple of $4$, then as long as one of $A_2$, $A_{-2}$ and $P_1$ is non-empty we can pick $i_2$ contained in one of them and decrease $u_{i_2}$ by $2$ (from $3$ to $1$ if $i_2 \in A_2$ or from $-1$ to $-3$ if $i_2 \in A_{-2}$, or carrying out the two reductions by $1$ represented by $i_2 \in P_1$). This updates the value of $d$ to $d-2$ a multiple of $4$. Otherwise, $\alpha_3 + \alpha_1+\alpha_{-1}+\alpha_{-3} \leq 2$ and $\alpha_2+\alpha_{-2}=0$, so $n\leq 2$. If $n=1$ then $k \in D_4$ so we set $\vec{u} = (k)$; we have forbidden $n=2$.

Now we pair up all $i \in P_1$, along with $i_2$ if it exists, into pairs $(j_1,j_2) \in P_2$ and promise that we will only reduce $u_{j_1}$ by $2$ if we are also reducing $u_{j_2}$ by $2$. With the possible exception of a single $i^{(2)}$ which cannot be paired, this means that all of these indices can now be treated as new indices $j \in P_2$ which permit a reduction of $\vec{u}$ by $4$.

Now $d$ can be reduced by $4$ by any of the following changes to $\vec{u}$:

\begin{itemize}
\item reducing $u_i$ from $2$ to $-2$ for $i \in A_{-1}$ in the positive case or $i \in A_1$ in the negative case;
\item reducing $u_i$ from $3$ to $-1$ for $i \in A_2$;
\item reducing $u_i$ from $3$ to $-1$ or from $2$ to $-2$ for $i \in A_1$ in the positive case, or from $2$ to $-2$ or from $1$ to $-3$ for $i \in A_{-1}$ in the negative case;
\item reducing $u_i$ from $1$ to $-3$ for $i \in A_{-2}$;
\item reducing $u_j$ by $4$ for $j \in P_2$.
\end{itemize}

After completing all of these reductions, and further reducing $u_{i^{(1)}}$ by $1$ if it exists and $u_{i^{(2)}}$ by $2$ if it exists, we must be in a case where $u_j = u_{\min}(i)$ for $j \in A_i$. At this point, $\sum_i u_i \leq k$. Now if we had not carried out the reductions of $u_{i^{(1)}}$ and  $u_{i^{(2)}}$, we would have had a value at most three more than this; so in this case we would have $\sum_i u_i \leq k+3$. on the other hand, we would never have deviated from $d$ being a multiple of $4$, so we would have $\sum_i u_i \equiv k \pmod 4$. So indeed $\sum_i u_i \leq k$ and we must have had some stage in the reductions of $d$ where $d=0$; for this $\vec{u}$ we have $d_1(\vec{u})=k$, $d_1(\vec{e})=m$, and $\vec{e}-\vec{u} \in D_4^n$ with $d_1(\vec{e}-\vec{u})=m-k$. Thus by induction we know that $f' \circ \tau_A(\vec{e}-\vec{u})=X_-$ and $f(\vec{u})=X_-$, so $f' \circ \tau_A(\vec{e})=X_-$ as required.

Now by induction we have proved that $f'(\vec{\alpha})=X_-$ for all $d_1(\vec{\alpha})<0$ whenever $n \neq 2$.
\end{proof}

The case of $n=2$ is special; here we know from Corollary \ref{m1} that for $\vec{\alpha}$ equal to $(0,0,1,0,1,0)$ and $(0,1,0,0,0,1)$ we have $f'(\vec{\alpha})=X_-$. Then setting $w=(1,1)$ and $M$ to be:
\begin{itemize}
\item $\left(\begin{array}{c c} +1 & -2 \\ -2 & +1 \end{array}\right)$ gives that $f'((0,0,0,2,0,0)) = X_-$;
\item $\left(\begin{array}{c c} -1 & -1 \\ +2 & -2 \end{array}\right)$ gives that $f'((0,0,1,0,0,1)) = X_-$;
\item $\left(\begin{array}{c c} +1 & -3 \\ -2 & +1 \end{array}\right)$ gives that $f'((0,0,0,1,1,0)) = X_-$;
\item $\left(\begin{array}{c c} +1 & -3 \\ -3 & +1 \end{array}\right)$ gives that $f'((0,0,0,0,2,0)) = X_-$;
\item $\left(\begin{array}{c c} -2 & -2 \\ +1 & -1 \end{array}\right)$ gives that $f'((0,0,0,1,0,1)) = X_-$;
\item $\left(\begin{array}{c c} -1 & -1 \\ -1 & -2 \end{array}\right)$ gives that $f'((0,0,0,0,1,1)) = X_-$; and
\item $\left(\begin{array}{c c} -1 & -1 \\ -2 & -2 \end{array}\right)$ gives that $f'((0,0,0,0,0,2)) = X_-$
\end{itemize}
Note that in many of these cases, $\vec{\alpha}(w,M_i)$ is a vector given by a previous item in this list. This addresses all cases of $d_1(\vec{\alpha})<0$ in the case of $n=2$. Hence for all $n$, $f'$ is constant on all $\vec{\alpha}$ with $d_1(\vec{\alpha}) <0$.

We say that $A'_-$ is the subset of $A'$ for which $d_1(\vec{\alpha})<0$.

\begin{cor}\label{a4}
$f'(\vec{\alpha})$ is constant over all $\vec{\alpha} \in A^-_{4,n}$.
\end{cor}

All that remains is to extend this to the case of $k>4$. We now have a domain for $f'$ of $A'_k$; but we know that $f'$ agrees with a Borda rule on $A'_4 \subset A'_k$. The following Lemma completes the proof entirely:

\begin{lem}\label{ak}
For all $k \geq 4$, $f' = T$ on $B'_k$ and $f'$ is constant on vectors $\vec{\alpha} \in A'_{k,-}$.
\end{lem}

\begin{proof}
We will show that for $k>4$ and any $n$, a vector $\vec{\alpha} \in A'_k \setminus A'_{k-1}$ with $b = d_1(\vec{\alpha}) \leq 0$ can be written as $\tau_A(\vec{u}+\vec{v})$ where $\tau_A(\vec{u}), \tau_A(\vec{v}) \in A'_{k-1}$ and $b \leq d_1 \circ \tau_A(\vec{u}), d_1 \circ \tau_A(\vec{v}) \leq 0$.

Take such a $\vec{\alpha}$. Then for $\varepsilon \in \pm 1$ we have $\alpha_{\varepsilon k} \geq 1$. Let $\vec{e}$ be a vector in $D_k^n$ with $\tau_A(\vec{e}) = \vec{\alpha}$, and define the sets $A_i = \set{j: e_i = j}$ so that the $A_i$ partition $\set{1, \dots, n}$ and $|A_i| = \alpha_i$.

Define $u_{\max}$ and $u_{\min}$ as follows, for $t>2$:

\begin{equation}\begin{aligned}
\begin{matrix}
i & u_{\max}(i) & u_{\min}(i) \\
\hline
+t & +t-1 & +1 \\
+2 & +1 & +1 \\
+1 & +2 & -1 \\
-1 & +1 & -2 \\
-2 & -1 & -1 \\
-t & -1 & -t+1
\end{matrix}
\end{aligned}\end{equation}

Now $u_{\max}(i)\geq i/2 \geq u_{\min}(i)$ for all $i$, so $\vec{u}$ defined by $u_i = u_{\max}(j)$ for all $i \in A_j$ has $\sum_i u_i \geq b/2$. We pick $i_1 \in A_{\varepsilon k}$. Now we incrementally decrease $\sum_i u_i$ by taking $i \neq i_1$ and changing $u_i$ from $u_{\max}(j)$ to $u_{\min}(j)$ where $i \in A_j$ and $|j| \leq 1$, and decreasing $u_i$ by $1$ repeatedly where $i \in A_j$ and $|j| >2$. As we do so, $\sum_i u_i$ decreases in intervals of $1$ and $3$ until $u_i = u_{\min}(j)$ for $i \neq i_1$ and $u_{i_1} = u_{\max}(\varepsilon k)$. Now $\sum_i u_i +2 - k = \sum_j \alpha_j u_{\max}(j) \leq b/2$, so $\sum_i u_i \leq b/2 + k -2$.

Now at some stage in this process was the last time that $\sum_i u_i$ was at least $b/2$. If this was the last step, then $b/2 < \sum_i u_i \leq b/2 + k-2$. Otherwise, $b/2 \leq \sum_i u_i \leq b/2 + 2$. In either case, we can reduce $u_{i_1}$ by $\sum_i u_i - \lceil b/2 \rceil \leq k-2$ (as $2 \leq k-2$) and get $\sum_i u_i = \lceil b/2 \rceil$.

Now clearly $\vec{u}, \vec{v} \in D_{k-1}^n$, $d_1 \circ \tau_A(\vec{u}) = \lceil b/2 \rceil$ and $d_1 \circ \tau_A(\vec{v}) = \lfloor b/2 \rfloor$, so they meet the conditions required.

We proceed to prove the Lemma by induction. We know that $f'=T$ on $B'_4$ by Corollary \ref{b4}. Suppose that $f' = T$ on $B'_j$ for $j \leq k-1$. Then for $\vec{\alpha} \in B'_k \setminus B'_{k-1}$, we find $\vec{u}$ and $\vec{v}$ as above. The inductive hypothesis gives us that $f' \circ \tau_A(\vec{u}) = f' \circ \tau_A(\vec{v}) = T$, so indeed $f' \circ \tau_A(\vec{u}+\vec{v}) = f'(\vec{\alpha}) = T$. This is true for all $\vec{\alpha} \in B'_k \setminus B'_{k-1}$, so combining with the inductive hypothesis that $f' = T$ on $B'_{k-1}$, we find that $f'=T$ on $B'_k$ as required.

We also know that $f' = X_-$ on $A'_-$ by Corollary \ref{a4}. Suppose that $f'=X_-$ on $A'_{j,-}$ for $j \leq k-1$. Then for $\vec{\alpha} \in A'_{k,-} \setminus A'_{k-1,-}$, we find $\vec{u}$ and $\vec{v}$ as above. The inductive hypothesis gives us that $f' \circ \tau_A(\vec{u}) \in \set{T, X_-}$ and $f' \circ \tau_A(\vec{v}) = X_-$, so indeed $f' \circ \tau_A(\vec{u}+\vec{v}) = f'(\vec{\alpha}) = X_-$. This is true for all $\vec{\alpha} \in A'_{k,-} \setminus A'_{k-1,-}$, so combining with the inductive hypothesis that $f' = X_-$ on $A'_{k-1,-}$, we find that $f'=X_-$ on $A'_{k,-}$ as required.
\end{proof}

Now all that's left is to observe that for $\vec{\alpha} \in A'_k$ with $d_1(\vec{\alpha}) >0$, we have $d_1(\varrho(\vec{\alpha}))<0$, so $f'(\varrho(\vec{\alpha})) = X_-$. Thus $f'(\vec{\alpha}) = -X_-$. So $f'$ must be identical to one of the three Borda rules over all of $A'_k$; either $X_-=L$, giving the positive Borda rule, or $X_- = W$, giving the negative Borda rule, or $X_- = T$, giving the tying rule.

\begin{thm}[Maskin's Theorem on $k>3$ candidates and finite voters]\label{FUAk}
For any SWF $F$ on $V$ finite and $C = \set{c_i}_{i \leq k}$ with $k>3$ satisfying U, MIIA, A and N:
\begin{enumerate}
\item If $F$ satisfies PR then it is the positive unweighted Borda rule or the tie rule.
\item If $F$ satisfies P then it is the positive unweighted Borda rule.
\end{enumerate}
\end{thm}

This completes our investigation of the MIIA condition in the case of complete anonymity between voters and $k>3$.

\printbibliography

\end{document}